\documentclass[11pt]{article}
\usepackage{amscd}
\usepackage{amsmath}
\usepackage{latexsym}
\usepackage{amsfonts}
\usepackage{amssymb}
\usepackage{amsthm}
\usepackage{graphicx}
\usepackage{verbatim}
\usepackage{mathrsfs}
\usepackage{enumerate}
\usepackage{hyperref}

\theoremstyle{plain}
\theoremstyle{definition}\newtheorem{theorem}{Theorem}[section]
\theoremstyle{plain}\newtheorem{lemma}[theorem]{Lemma}
\theoremstyle{plain}
\theoremstyle{plain}\newtheorem{proposition}[theorem]{Proposition}
\theoremstyle{remark}\newtheorem{remark}{Remark}[section]

\newcommand{\norm}[1]{\left\|#1\right\|}

\begin{document}
\title{Global well-posedness for a modified critical dissipative quasi-geostrophic equation}
\author{Changxing Miao\footnote{Institute of Applied Physics and Computational Mathematics, P.O. Box 8009,
            Beijing 100088, P.R. China. Email: miao\_{}changxing@iapcm.ac.cn.}\;  and  Liutang Xue\footnote{
            The Graduate School of China Academy of Engineering Physics, P.O. Box 2101, Beijing 100088, P.R. China. Email: xue\_{}lt@163.com.}}
\date{}
\maketitle
\begin{abstract}
In this paper we consider the following modified quasi-geostrophic equation
\begin{equation*}
 \partial_{t}\theta+u\cdot\nabla\theta+\nu |D|^{\alpha}\theta=0,
 \quad
 u=|D|^{\alpha-1}\mathcal{R}^{\bot}\theta,\quad x\in\mathbb{R}^2
\end{equation*}
with $\nu>0$ and $\alpha\in ]0,1[\,\cup \,]1,2[$. When $\alpha\in]0,1[$, the equation was firstly introduced
by Constantin, Iyer and Wu in \cite{ref ConstanIW}.
Here, by using the modulus of continuity method, we prove the global well-posedness of the system with the smooth initial data.
As a byproduct, we also show that for every $\alpha\in ]0,2[$, the Lipschitz norm of the solution has a uniform exponential bound.
\end{abstract}


\section{Introduction}
\setcounter{section}{1}\setcounter{equation}{0}

 In this paper we
focus on the following modified 2D dissipative
quasi-geostrophic equation
\begin{equation}\label{MQG}
   \begin{cases}
   \partial_{t}\theta+u\cdot\nabla\theta+\nu |D|^{\alpha}\theta=0 \\
   u=|D|^{\alpha-1}\mathcal{R}^{\bot}\theta, \qquad  \theta|_{t=0}=\theta_{0}(x)
   \end{cases}
\end{equation}
with $\nu>0$, $\alpha\in ]0,1[\,\cup \,]1,2[$,
$|D|^{\beta}=(-\Delta)^{\frac{\beta}{2}}$ is defined via the
Fourier transform
$$\widehat{|D|^{\beta}f}(\zeta)=|\zeta|^{\beta}\hat{f}(\zeta)$$
and
$$\mathcal{R}^{\bot}\theta=(-\mathcal{R}_{2}\theta,\mathcal{R}_{1}\theta)=|D|^{-1}( \partial_{2}\theta,-\partial_{1}\theta)$$
where $\mathcal{R}_{i}$($i=1,2$) are the usual Riesz transforms (cf. \cite{ref Duo}).

When $\alpha=0$, this model describes the evolution of the vorticity
of a two dimensional damped inviscid incompressible fluid. The case
of $\alpha=1$ just is the critical dissipative quasi-geostrophic
equation which arises in the geostrophic study of rotating fluids (cf. \cite{Con1}).
Although when $\alpha=2$ the flow term in \eqref{MQG} vanishes, we can still view the model introduced in \cite{ref FRIV}
as a meaningful generalization of this endpoint case, where the model is derived from the study of
the full magnetohydrodynamic equations and the divergence-free three-dimensional velocity $u$ satisfies $u= M [\theta]$
with $ M $ a nonlocal differential operator of order 1.

For convenience, we here recall the well-known 2D quasi-geostrophic
equation
\begin{equation*}
(QG)_{\alpha}\quad\begin{cases}
   \partial_{t}\theta+u\cdot\nabla\theta+\nu |D|^{\alpha}\theta=0 \\
   u=\mathcal{R}^{\bot}\theta, \qquad  \theta(0,x)=\theta_{0}(x)
   \end{cases}
\end{equation*}
where $\nu\geq 0$ and $0\leq\alpha\leq2$. When $\nu>0$, $\alpha\in ]0,1[\,\cup \, ]1,2[$, we
observe that the system \eqref{MQG} is almost the same with the quasi-geostrophic equation, and its only difference
lies on introducing an extra $|D|^{\alpha-1}$ in the definition of
$u$. When $\alpha\in ]0,1[$, $|D|^{\alpha-1}$ is a negative derivative operator and always plays a good role; while when $\alpha\in]1,2[$,
$|D|^{\alpha-1}$ is a positive derivative operator and always takes a bad part. Moreover, corresponding to the dissipation
operator $|D|^\alpha$, this additional operator makes the equation $(QG)_\alpha$ be a new balanced state: the flow term
$u\cdot\nabla\theta$ scale the same way as the dissipative term
$|D|^{\alpha}\theta$, i.e., the equation \eqref{MQG} is scaling
invariant under the transformation
$$\theta(t,x)\rightarrow \theta_{\lambda}(t,x):=\theta(\lambda^{\alpha}t,\lambda x),\quad \mathrm{with}\quad \lambda>0.$$
We note that in the sense of scaling invariance, the
system \eqref{MQG} is similar to the critical dissipative
quasi-geostrophic equation.

Recently, when $\alpha\in ]0,1[$, Constantin, Iyer and Wu in \cite{ref ConstanIW} introduced
this modified quasi-geostrophic equation and proved the global
regularity of Leray-Hopf weak solutions to the system with $L^{2}$ initial
data. Basically, they use the method from Caffarelli-Vasseur
\cite{Caffarelli} which deals with the same issue of 2D critical
dissipative quasi-geostrophic equation $(QG)_{1}$. We also remark
that partially because of its simple form and its internal analogy with the 3D
Euler/Navier-Stokes equations, the quasi-geostrophic equation $(QG)_\alpha$, especially the critical one $(QG)_{1}$, has been
extensively considered (see e.g. \cite{ref AbidiH,Caffarelli, ref ChenMZ, Con1, Con2, Con3, ref AC-DC, ref Dong, KisNV,Wu1} and references therein).
While global regularity of Navier-Stokes
equations remains an outstanding challenge in mathematical physics,
the global issue of the 2D critical dissipative
quasi-geostrophic equation has been in a satisfactory state.
In \cite{Con3} Constantin, Cordoba and Wu showed the global well-posedness
of the classical solution under the condition that the zero-dimensional $L^\infty$ norm of the data is small.
This smallness assumption was firstly removed by Kiselev, Nazarov and Volberg in \cite{KisNV}, where they obtained the global
well-posedness for the arbitrary periodic smooth initial data by using a modulus of continuity method.
Almost at the same time, Caffarelli and Vasseur in \cite{Caffarelli} resolved the
problem to establish the global regularity of weak solutions
associated with $L^{2}$ initial data by exploiting the De Giorgi method. We also cite the work of
Abidi-Hmidi \cite{ref AbidiH} and Dong-Du \cite{ref Dong}, as
extended work of \cite{KisNV}, in which the authors proved the
global well-posedness with the initial data belonging to the
(critical) space $\dot{B}^{0}_{\infty,1}$ and $H^{1}$ respectively
without the additional periodic assumption.

The main goal in this paper is to prove the global well-posedness of
the smooth solutions for the system \eqref{MQG} with $\alpha\in ]0,1[\,\cup\,]1,2[$. In contrast with
the work of \cite{ref ConstanIW}, we here basically follow the
pathway of \cite{KisNV} to obtain the global results by
constructing suitable moduli of continuity. Precisely, we have

\begin{theorem}\label{main thm1}
Let $\nu>0$, $\alpha\in]0,2[$ and $\theta_{0}\in H^{m}$, $m>2$, then
there exists a unique global solution
\begin{equation*}
 \theta\in\mathcal{C}([0,\infty[;H^{m}) \cap L_{\mathrm{loc}}^2([0,\infty[; H^{m+\frac{\alpha}{2}})\cap \mathcal{C}^{\infty}(]0,\infty[\times \mathbb{R}^{2})
\end{equation*}
to the modified
quasi-geostrophic equation \eqref{MQG}. Moreover, we get the uniform bound of the
Lipschitz norm
\begin{equation}\label{eq ExpBdd}
 \sup_{t\geq 0}\norm{\nabla\theta(t)}_{L^\infty}\leq C\norm{\nabla\theta_0}_{L^\infty}\exp\{C\norm{\theta_0}_{L^\infty}\},
\end{equation}
where $C$ is an absolute constant depending only on $\alpha,\nu$.
\end{theorem}

The proof is divided into two parts. First through applying the
classical method, we obtain the local existence results (Proposition \ref{prop local}) and
further build the blowup criterion (Proposition \ref{prop break}). Then we adopt the nonlocal
maximum principle method of Kiselev-Nazarov-Volberg and finally
manage to remove all the possible breakdown scenarios by
constructing suitable moduli of continuity.

\begin{remark}\label{rem main1}
 The main new ingredient in the global existence consists of
 modulus of continuity with the explicit formula \eqref{eq modulus0}
 which is suitable for  every $\alpha\in]0,2[$.  This MOC
 has a logarithmic growth near infinity, and further yields the uniform exponential bound of the Lipschitz norm of the solution.
 In particular, when $\alpha=1$, \eqref{eq ExpBdd} is a slight improvement of the corresponding bound in \cite{Kis}, where it is a double exponential type.
\end{remark}

\begin{remark}\label{rem main2}
 When $\alpha\in ]1,2[$, from the viewpoint of weak solutions, the authors in \cite{MiaoXue2} find that the regularity criterion of \eqref{MQG}
 in terms of H\"older continuous solutions is worse than
 that of $(QG)_1$, that is, we a priori need $\theta\in L^{\infty}([t_0,t_1]; C^\sigma(\mathbb{R}^2))$ with $\sigma >\frac{\alpha-1}{2}$
 to ensure that $\theta$ is a smooth solution in $]t_0,t_1]$ (in contrast with $\sigma>0$ when $\alpha=1$).
 Thus if we rely on this criterion, it is not sufficient to obtain the global regularity of
 \eqref{MQG} with $\alpha\in ]1,2[$ by merely applying the method of \cite{Caffarelli}.
\end{remark}

The paper is organized as follows. In Section 2, we present some
preparatory results. In Section 3, some facts about modulus of
continuity are discussed. In Section 4, we obtain the local results
and establish blowup criterion. Finally, we prove the global
existence in Section 5.

\section{Preliminaries}
\setcounter{section}{2}\setcounter{equation}{0}

In this preparatory section, we present the definitions and some
related results of the Sobolev spaces and the Besov spaces, also we
provide some important estimates which will be used later.

We begin by introducing some notations.
\\
$\diamond$ Throughout this paper $C$ stands for a constant which may be different from line to line. We
sometimes use $A\lesssim B$ instead of $A\leq C B$, and use $A\lesssim_{\beta,\gamma\cdots}B$ instead of $A\leq C(\beta,\gamma,\cdots)B$
with $C(\beta,\gamma,\cdots)$ a constant depending on $\beta,\gamma,\cdots$. For $A\thickapprox B$ we mean $A\lesssim B\lesssim A$.
\\
$\diamond$ Denote by
$\mathcal{S}(\mathbb{R}^{n})$ the Schwartz space of rapidly
decreasing smooth functions, $\mathcal{S}'(\mathbb{R}^{n})$ the
space of tempered distributions,
$\mathcal{S}'(\mathbb{R}^{n})/\mathcal{P}(\mathbb{R}^{n})$ the
quotient space of tempered distributions which modulo polynomials.
\\
$\diamond$
$\mathcal{F}f$ or $\hat{f}$ denotes the Fourier transform, that is
$\mathcal{F}f(\zeta)=\hat{f}(\zeta)=\int_{\mathbb{R}^{n}}e^{-ix\cdot\zeta}f(x)\textrm{d} x,$
while $\mathcal{F}^{-1}f$ the inverse Fourier transform, namely,
$\mathcal{F}^{-1}f(x)=(2\pi)^{-n}\int_{\mathbb{R}^{n}}e^{ix\cdot\zeta}f(\zeta)\textrm{d}
\zeta$.

Now we give the definition of $L^{2}$ based Sobolev space. For
$s\in\mathbb{R}$, the inhomogeneous Sobolev space
\begin{equation*}
H^{s}:=\Big\{f\in \mathcal{S}'(\mathbb{R}^{n});
\norm{f}^{2}_{H^{s}}:=\int_{\mathbb{R}^{n}}(1+|\zeta|^{2})^{s}|\hat{f}(\zeta)|^{2}\textrm{d}
\zeta<\infty\Big\}
\end{equation*}
Also one can define the corresponding homogeneous space:
\begin{equation*}
\dot{H}^{s}:=\Big\{f\in
\mathcal{S}'(\mathbb{R}^{n})/\mathcal{P}(\mathbb{R}^{n});
\norm{f}^{2}_{\dot{H}^{s}}:=\int_{\mathbb{R}^{n}}|\zeta|^{2s}|\hat{f}(\zeta)|^{2}\textrm{d}
\zeta<\infty\Big\}
\end{equation*}

The following calculus inequality is well-known(see \cite{ref
Bertozzi-Majda})
\begin{lemma}\label{lem SobCalIneq}
$\forall m\in\mathbb{R}^{+}$, there exists a constant $c_{m}>0$ such  that
   \begin{equation}\label{CalInqSob1}
    \norm{fg}_{H^{m}}\leq c_{m}
    \big(\norm{f}_{L^{\infty}}\norm{g}_{H^{m}}+\norm{f}_{H^{m}}\norm{g}_{L^{\infty}}\big).
   \end{equation}
\end{lemma}

To define Besov space we need the following dyadic partition of unity
(see e.g. \cite{ref chemin}). Choose two nonnegative radial
functions $\chi$, $\varphi\in \mathcal{D}(\mathbb{R}^{n})$ be
supported respectively in the ball $\{\zeta\in
\mathbb{R}^{n}:|\zeta|\leq \frac{4}{3} \}$ and the shell $\{\zeta\in
\mathbb{R}^{n}: \frac{3}{4}\leq |\zeta|\leq
  \frac{8}{3} \}$ such that
\begin{equation*}
 \chi(\zeta)+\sum_{j\geq 0}\varphi(2^{-j}\zeta)=1, \quad
 \forall\zeta\in\mathbb{R}^{n}; \qquad
 \sum_{j\in \mathbb{Z}}\varphi(2^{-j}\zeta)=1, \quad \forall\zeta\neq 0.
\end{equation*}
For all $f\in\mathcal{S}'(\mathbb{R}^{n})$ we define the
nonhomogeneous Littlewood-Paley operators
\begin{equation*}
 \Delta_{-1}f:= \chi(D)f; \;\;
 \Delta_{j}f:= \varphi(2^{-j}D)f,\; S_j f:=\sum_{-1\leq k\leq j-1}\Delta_k f,\quad \forall j\in\mathbb{N},
\end{equation*}
And the homogeneous Littlewood-Paley operators can be defined as follows
\begin{equation*}
  \dot{\Delta}_{j}f:= \varphi(2^{-j}D)f;\; \dot S_j f:= \sum_{k\in\mathbb{Z},k\leq j-1}\dot \Delta_k, f\quad \forall j\in\mathbb{Z}.\quad
\end{equation*}

Now we introduce the definition of Besov spaces . Let $(p,r)\in
[1,\infty]^{2}$, $s\in\mathbb{R}$, the nonhomogeneous Besov space
\begin{equation*}
  B^{s}_{p,r}:=\Big\{f\in\mathcal{S}'(\mathbb{R}^{n});\norm{f}_{B^{s}_{p,r}}:=\norm{\{2^{js}\|\Delta
  _{j}f\|_{L^{p}}\}_{j\geq -1}}_{\ell^{r}}<\infty  \Big\}
\end{equation*}
and the homogeneous space
\begin{equation*}
  \dot{B}^{s}_{p,r}:=\Big\{f\in\mathcal{S}'(\mathbb{R}^{n})/\mathcal{P}(\mathbb{R}^{n});
  \norm{f}_{\dot{B}^{s}_{p,r}}:=\|\{2^{js}\|\dot{\Delta}
  _{j}f\|_{L^{p}}\}_{j\in\mathbb{Z}}\|_{\ell^{r}(\mathbb{Z})}<\infty  \Big\}.
\end{equation*}
We point out that for all $s\in\mathbb{R}$, $B^{s}_{2,2}=H^{s}$ and
$\dot{B}^{s}_{2,2}=\dot{H}^{s}$.

The classical space-time Besov space $L^{\rho}([0,T],B^{s}_{p,r})$, abbreviated by
$L^{\rho}_{T}B^{s}_{p,r}$, is the set of tempered distribution $f$
such that
\begin{equation*}
  \norm{f}_{L^{\rho}_{T}B^{s}_{p,r}}:=\norm{\norm{\{2^{js}\norm{\Delta_{j}f}_{L^{p}}\}_{j\geq -1}}_{\ell^{r}}}_{L^{\rho}([0,T])}<\infty.
\end{equation*}
We can similarly extend to the homogeneous one $L^{\rho}_{T}\dot{B}^{s}_{p,r}$.

Bernstein's inequality is fundamental in the analysis involving
Besov spaces (see \cite{ref chemin})
\begin{lemma}
Let $f\in L^{a}$, $1\leq a\leq b\leq \infty$. Then for every $
(k,q)\in\mathbb{N}^{2}$ there exists a constant $C>0$ such that
\begin{equation*}
 \sup_{|\alpha|=k}\norm{\partial^{\alpha} S_{q}f}_{L^{b}}\leq C 2
 ^{q(k+n(\frac{1}{a}-\frac{1}{b}))}\norm{f}_{L^{a}},
\end{equation*}
\begin{equation*}
 C^{-1}2
 ^{q k}\norm{f}_{L^{a}}\leq \sup_{|\alpha|=k}\norm{\partial^{\alpha}
 \Delta_{q}f}_{L^{a}}\leq C 2^{qk}\norm{f}_{L^{a}}
 \end{equation*}
\end{lemma}

Finally we state an important maximum principle for the transport-diffusion equation (cf. \cite{ref AC-DC})
\begin{proposition}\label{prop MP}
Let $u$ be a smooth divergence-free vector field and $f$ be a smooth
function. Assume that $\theta$ is the smooth solution of the equation
\begin{equation*}
  \partial_{t}\theta + u\cdot\nabla\theta+\nu |D|^{\alpha}\theta=f,\quad \mathrm{div}u =0,
\end{equation*}
with initial datum $\theta_0$ and $\nu\geq 0$, $0\leq\alpha\leq2$, then for every $p\in
[1,\infty]$ we have
\begin{equation}\label{eq TDMaxPrin}
 \norm{\theta(t)}_{L^{p}}\leq
 \norm{\theta_{0}}_{L^{p}}+\int^{t}_{0}\norm{f(\tau)}_{L^{p}}
 \,\textrm{d}\tau.
\end{equation}
\end{proposition}


\section{Moduli of Continuity}\label{sec MOC}
\setcounter{section}{3}\setcounter{equation}{0}

In this section, we discuss the moduli of continuity which play a
key role in our global existence part.

We suppose that $\omega$ is a modulus of continuity, that is, a
continuous, increasing, concave function on $[0,\infty)$ such that
$\omega(0)=0$. We say that a function
$f:\mathbb{R}^{n}\rightarrow\mathbb{R}^{m}$ has modulus of
continuity if $|f(x)-f(y)|\leq \omega(|x-y|)$ for all
$x,y\in\mathbb{R}^{n}$ and that $f$ has strict modulus of continuity
if the inequality is strict for $x\neq y$.

Next we introduce the pseudo-differential operators $\mathcal{R}_{\alpha,j}$
which may be termed as the modified Riesz transforms
\begin{proposition}\label{prop Fmula}
Let $\alpha\in ]0,2[$, $1\leq j\leq n$, $n\geq 2$, then for every $f\in \mathcal{S}(\mathbb{R}^{n})$
\begin{equation}\label{Fmula R_al}
 \mathcal{R}_{\alpha,j}f(x)= |D|^{\alpha-1}\mathcal{R}_{j}f(x)=
 c_{\alpha,n} \mathrm{p.v.}\int_{\mathbb{R}^{n}}\frac{y_{j}}{|y|^{n+\alpha}}f(x-y)\,\textrm{d}
 y,
\end{equation}
where $c_{\alpha,n}$ is the normalization constant such that
\begin{equation*}
\widehat{\mathcal{R}_{\alpha,j}f}(\zeta)=-i\frac{\zeta_{j}}{|\zeta|^{2-\alpha}}\hat{f}(\zeta).
\end{equation*}
\end{proposition}

The proof is placed in the appendix. Also note that when $\alpha\in ]0,1[$, we do not need to introduce the principle value of
integral expression in the formula \eqref{Fmula R_al}.

The pseudo-differential operators like the modified Riesz transforms
do not preserve the moduli of continuity generally, but they also do not
destroy them too much either. Precisely, similarly as the Lemma in \cite{KisNV}, we have
\begin{lemma}\label{lem ModuliModRiez}
If the function $\theta$ has the modulus of continuity $\omega$,
then $u=(-\mathcal{R}_{\alpha,2}\theta,\mathcal{R}_{\alpha,1}\theta)$ ($\alpha\in]0,2[$) has the modulus of
continuity
\begin{equation}\label{eq ModuliMRiez}
 \Omega(\xi)=A_{\alpha}\bigg(\int^{\xi}_{0}\frac{\omega(\eta)}{\eta^{\alpha}}\textrm{d} \eta+
 \xi\int_{\xi}^{\infty}\frac{\omega(\eta)}{\eta^{1+\alpha}}\textrm{d}
 \eta\bigg)
\end{equation}
with some absolute constant $A_{\alpha}>0$ that may depend on $\alpha$.
\end{lemma}
\begin{proof}
The modified Riesz transforms are pseudo-differential operators with
kernels $K(x)=\frac{S(x')}{|x|^{n-1+\alpha}}$ (in our special case,
$n=2$ and $S(x')=\frac{x_{j}}{|x|}, j=1,2$), where
$x'=\frac{x}{|x|}\in \mathbb{S}^{n-1} $. The function $S\in
C^{1}(\mathbb{S}^{n-1})$ and $\int_{\mathbb{S}^{n-1}}S(x')d
\sigma(x')=0$. Assume that the function
$f:\mathbb{R}^{n}\rightarrow\mathbb{R}^{m}$ has some modulus of
continuity $\omega$, that is $|f(x)-f(y)|\leq\omega(|x-y|)$ for all
$x,y\in \mathbb{R}^{n}$. Then take any $x,y$ with $|x-y|=\xi$, and
consider the difference
\begin{equation}\label{differ}
 \int K(x-t)f(t)\textrm{d} t-\int K(y-t)f(t)\textrm{d} t.
\end{equation}
First due to the cancelation property of $S$ we have
\begin{equation*}
 \bigg|\int_{|x-t|\leq 2\xi}K(x-t)f(t)\textrm{d} t\bigg|=
 \bigg|\int_{|x-t|\leq 2\xi}K(x-t)(f(t)-f(x))\textrm{d} t\bigg|\leq
 C\int_{0}^{2\xi}\frac{\omega(r)}{r^{\alpha}}\textrm{d} r
\end{equation*}
since $\omega$ is concave, we obtain
\begin{equation}\label{control}
 \int_{0}^{2\xi}\frac{\omega(r)}{r^{\alpha}}\textrm{d} r\leq 2^{2-\alpha}\int_{0}^{\xi}\frac{\omega(r)}{r^{\alpha}}\textrm{d} r
\end{equation}
A similar estimate holds for the second integral in \eqref{differ}.
Next, set $z=\frac{x+y}{2}$, then
\begin{equation*}
\begin{split}
 &\bigg|\int_{|x-t|\geq2\xi}K(x-t)f(t)\textrm{d} t-\int_{|y-t|\geq2\xi}K(y-t)f(t)\textrm{d}
 t\bigg| \\
 &=\bigg|\int_{|x-t|\geq2\xi}K(x-t)(f(t)-f(z))\textrm{d} t-\int_{|y-t|\geq2\xi}K(y-t)(f(t)-f(z))\textrm{d}
 t\bigg| \\
 &\leq \int_{|z-t|\geq 3\xi}|K(x-t)-K(y-t)||f(t)-f(z)|\textrm{d} t  \\
 &\quad+\int_{\frac{3\xi}{2}\leq|z-t|\leq 3\xi}(|K(x-t)|+|K(y-t)|)|f(t)-f(z)|\textrm{d}
 t\\
 &=I_{1}+I_{2}
\end{split}
\end{equation*}
To estimate the first integral, we use the smoothness condition of
$S$ to get
\begin{equation*}
|K(x-t)-K(y-t)|\leq C \frac{|x-y|}{|z-t|^{n+\alpha}} \quad
\text{when} \, |z-t|\geq 3\xi
\end{equation*}
thus
\begin{equation*}
I_{1}\leq
C\xi\int_{3\xi}^{\infty}\frac{\omega(r)}{r^{1+\alpha}}\textrm{d} r
\leq
C3^{-\alpha}\xi\int_{\xi}^{\infty}\frac{\omega(3r)}{r^{1+\alpha}}\textrm{d}
r\leq
C\xi\int_{\xi}^{\infty}\frac{\omega(r)}{r^{1+\alpha}}\textrm{d} r
\end{equation*}
For the second integral, using the concavity of $\omega$ and
\eqref{control}, we have
\begin{equation*}
\begin{split}
I_{2}\leq&
2C\omega(3\xi)\xi^{1-\alpha}\int_{\xi\leq|x-t|\leq\frac{7}{2}\xi}\frac{1}{|x-t|^{n}}\textrm{d}
t \\
\leq & C \omega(\xi)\xi^{1-\alpha}\leq C
2^{\alpha}\int_{\xi}^{2\xi}\frac{\omega(r)}{r^{\alpha}}\textrm{d} r
 \leq C\int_{0}^{\xi}\frac{\omega(r)}{r^{\alpha}}\textrm{d} r
\end{split}
\end{equation*}
\end{proof}

Now we consider the action of the fractional differential
operators $|D|^{\alpha}$($\alpha\in ]0,2[$) on the function having
modulus of continuity. Precisely,
\begin{lemma}\label{lem Modulidiffer}
 If the function $\theta:\mathbb{R}^{2}\rightarrow\mathbb{R}$ has modulus of continuity $\omega$,
 and especially satisfies $\theta(x)-\theta(y)=\omega(\xi)$
at some $x,y\in\mathbb{R}^{2}$ with $|x-y|=\xi>0$, then we have
\begin{equation}\label{eq modulidiffer}
\begin{split}
 \bigl[(-|D|^{\alpha})\theta\bigr](x)-\bigl[(-|D|^{\alpha})\theta\bigr](y)
 \leq B_{\alpha}& \int_{0}^{\frac{\xi}{2}}\frac{\omega(\xi+2\eta)+\omega(\xi-2\eta)-2\omega(\xi)}{\eta^{1+\alpha}}\textrm{d}
 \eta \\
 + & B_{\alpha}\int_{\frac{\xi}{2}}^{\infty}\frac{\omega(2\eta+\xi)-\omega(2\eta-\xi)-2\omega(\xi)}{\eta^{1+\alpha}}\textrm{d}
 \eta
\end{split}
\end{equation}
where $B_{\alpha}>0$ is an absolute constant.
\end{lemma}

\begin{remark}\label{lem modulidiffer}
In fact this result has occurred in \cite{ref Yu}, as a
generalization of the one in \cite{KisNV}. For convenience, we
prove it again for the general $n$-dimensional case and place the
proof in the appendix. Also note that due to concavity of $\omega$,
both terms on the righthand side of \eqref{eq modulidiffer} are strictly
negative.
\end{remark}

\section{Local existence and Blowup criterion}
\setcounter{section}{4}\setcounter{equation}{0}

Our purpose in this section is to prove the following local result
\begin{proposition}\label{prop local}
Let $\nu>0$, $0<\alpha<2$ and the initial data $\theta_{0}\in
H^{m}$, $m>2$. Then there exists a positive $T$ depending only on
$\alpha$, $\nu$ and $\norm{\theta_{0}}_{H^{m}}$ such that the modified
quasi-geostrophic equation \eqref{MQG} generates a unique solution
$\theta\in \mathcal{C}([0,T],H^{m})\cap L^2([0,T],H^{m+\frac{\alpha}{2}})$. Moreover we have $t^{\gamma}\theta\in
L^{\infty}(]0,T],H^{m+\gamma\alpha})$ for all $\gamma\geq 0$, which implies $\theta\in \mathcal{C}^{\infty}(]0,T]\times\mathbb{R}^2)$.
\end{proposition}
We further obtain the following criterion for the breakdown of
smooth solutions
\begin{proposition}\label{prop break}
Let $T^{*}$ be the maximal existence time of $\theta$ in
$\mathcal{C}([0,T^{*}),H^{m})\cap L^2([0,T^*), H^{m+\frac{\alpha}{2}})$. If $T^{*}<\infty$ then we necessarily have
\begin{equation}\label{eq blowup}
 \int_{0}^{T^{*}}\norm{\nabla \theta(t,\cdot)}_{L^{\infty}}^{\alpha}\textrm{d}
 t=\infty.
\end{equation}
\end{proposition}

The method of proof for the Proposition \ref{prop local} is to
regularize the equation \eqref{MQG} by the standard Friedrich method, and then pass to the limit for the regularization
parameter.

Denote the frequency cutoff operator $\mathcal{J}_{\epsilon}: L^2(\mathbb{R}^{2})\rightarrow
H^m(\mathbb{R}^{2})$, $\epsilon>0$, $m\geq 0$ by
\begin{equation*}
 (\mathcal{J}_{\epsilon}f)(x)
 =\mathcal{F}^{-1}(\hat{f}(\cdot) 1_{B_{1/\epsilon}}(\cdot) )(x)=
 (2\pi)^{-2}\int_{\mathbb{R}^2} e^{i x\cdot\zeta} \hat{f}(\zeta)1_{\{|\cdot|\leq \frac{1}{\epsilon}\}}(\zeta)\mathrm{d}\zeta.
\end{equation*}
The following properties of $\mathcal{J}_{\epsilon}$ are obvious.
\begin{lemma}\label{lem mollifier}
Let $\mathcal{J}_{\epsilon}$ be the projection operator defined as above,
$m\in \mathbb{R}^{+}$, $k\in \mathbb{R}^{+}$, $\delta\in[0,m[$. Then
\begin{enumerate}[(i)]
 \item
  for all $f\in H^m$, $\lim_{\epsilon\rightarrow 0}\norm{\mathcal{J}_{\epsilon}f-f}_{H^{m}}=0$.
 \item
  for all $f\in H^{m}$,
  $|D|^{m}(\mathcal{J}_{\epsilon}f)=\mathcal{J}_{\epsilon}(|D|^{m}f)$ and $\Delta_{j}(\mathcal{J}_{\epsilon}f)=\mathcal{J}_{\epsilon}(\Delta_{j}f)$.
 \item
   for all $ f\in H^{m}$,
   $\norm{\mathcal{J}_{\epsilon}f-f}_{H^{m-\delta}}\lesssim \epsilon^\delta
   \norm{f}_{H^{m}}$ and $\norm{\mathcal{J}_{\epsilon}f}_{H^{m+k}}\lesssim
  \frac{1}{\epsilon^{k}}\norm{f}_{H^{m}}$.
\end{enumerate}
\end{lemma}
Then we regularize the modified quasi-geostrophic equation
\eqref{MQG} as follows
\begin{equation}\label{ReMQG}
 \begin{cases}
    &\theta^{\epsilon}_{t}+\mathcal{J}_{\epsilon}
   \big((\mathcal{J}_\epsilon u^{\epsilon})\cdot\nabla(\mathcal{J}_{\epsilon}\theta^{\epsilon})\big)
   +\nu \mathcal{J}_{\epsilon}|D|^{\alpha}\theta^{\epsilon} =0
   \\
    &u^{\epsilon}=|D|^{\alpha-1}\mathcal{R}^{\perp}\theta^{\epsilon}, \quad
    \theta^{\epsilon}|_{t=0}=\mathcal{J}_\epsilon\theta_{0}.
 \end{cases}
\end{equation}
For this approximate system, we have
\begin{proposition}\label{prop RegGlobal}
Let the initial data $\theta_{0}\in L^2$. Then for any
$\epsilon>0$ there exists a unique global solution
$\theta^{\epsilon}\in \mathcal{C}^{1}([0,\infty[, H^{\infty})$ to the regularized
equation \eqref{ReMQG}.
\end{proposition}

\begin{proof}
  We can write \eqref{ReMQG} as follows
\begin{equation}
  \frac{d}{d t} \theta^\epsilon = F_\epsilon(\theta^\epsilon),  \quad  \theta^{\epsilon}|_{t=0}=\mathcal{J}_\epsilon\theta_{0},
\end{equation}
with
\begin{equation*}
  F_\epsilon(\theta^\epsilon) = - \mathcal{J}_\epsilon\big((\mathcal{J}_\epsilon u^{\epsilon})\cdot\nabla(\mathcal{J}_{\epsilon}\theta^{\epsilon})\big)
   -\nu \mathcal{J}_{\epsilon}|D|^{\alpha}\theta^{\epsilon}.
\end{equation*}
For every $\epsilon>0$, we can show that
\begin{equation*}
  \|F_\epsilon(f)\|_{L^2}\lesssim_{\epsilon,\nu} \|f\|_{L^2} +\|f\|^2_{L^2},
\end{equation*}
and
\begin{equation*}
  \|F_\epsilon(f_1,f_2)\|_{L^2}\lesssim_{\epsilon,\nu,\|f_i\|_{L^2}}\|f_1-f_2\|_{L^2},
\end{equation*}
where $f$, $f_1$, $f_2$ are all in $L^2$. This means that $F_\epsilon$ maps $L^2$ into $L^2$ and $F_\epsilon$ is locally Lipschitz continuous on $L^2$.
Hence the Cauchy-Lipschitz theorem ensures that for every $\theta_0\in L^2$, there exists a unique solution
$\theta^\epsilon\in\mathcal{C}^1([0,T_\epsilon[,L^2)$ with
$T_\epsilon>0$ is the maximus existence time.

Moreover, using the $L^2$ energy method, form $\mathrm{div}u^\epsilon=0$ and $\mathcal{J}_\epsilon\theta^\epsilon\in \mathcal{C}^1([0,T_\epsilon[, H^\infty)$,
we get
\begin{equation*}
  \frac{1}{2}\frac{d}{d t}\|\theta^\epsilon\|_{L^2}^2 + \||D|^{\alpha/2}\mathcal{J}_\epsilon \theta^\epsilon\|_{L^2}^2 = 0
\end{equation*}
Thus
\begin{equation*}
  \sup_{t\in [0,T_\epsilon[}\|\theta^\epsilon(t)\|_{L^2}\leq \|\mathcal{J}_\epsilon\theta_0\|_{L^2}\leq \|\theta_0\|_{L^2}.
\end{equation*}
Then the classical continuation criterion guarantees $T_\epsilon=\infty$.

Moreover, since $\mathcal{J}_\epsilon\theta^\epsilon$ is also a solution of \eqref{ReMQG}, form the uniqueness of $\theta^\epsilon$ we find
$\theta^\epsilon=\mathcal{J}_\epsilon \theta^\epsilon$.
\end{proof}

\begin{remark}
 From the proof we know $\theta^\epsilon=\mathcal{J}_\epsilon \theta^\epsilon$, thus \eqref{ReMQG} will be written as follows
 \begin{equation}\label{ReMQG2}
 \begin{cases}
  \begin{split}
    &\theta^{\epsilon}_{t}+\mathcal{J}_{\epsilon}
   (  u^{\epsilon}\cdot\nabla\theta^{\epsilon})
   +\nu |D|^{\alpha}\theta^{\epsilon} =0
   \\
    &u^{\epsilon}=|D|^{\alpha-1}\mathcal{R}^{\perp}\theta^{\epsilon}, \quad
    \theta^{\epsilon}|_{t=0}=\mathcal{J}_\epsilon\theta_{0}.
  \end{split}
 \end{cases}
 \end{equation}
 In the sequel we shall instead consider this form.
\end{remark}

Next, we prove the main result in this section.

\begin{proof}[Proof of Proposition \ref{prop local}]

\textit{Step 1: Uniform Bounds.}

We claim that: the regularized solution $\theta^{\epsilon}\in
\mathcal{C}^{1}([0,\infty[,H^{\infty})$ to equation \eqref{ReMQG} satisfies
\begin{equation}\label{eq HmEst1}
  \frac{d}{2dt}\norm{\theta^{\epsilon}}_{B^{m}_{2,2}}^{2}+\frac{
  \nu}{2}\norm{|D|^{\frac{\alpha}{2}}\theta^{\epsilon}}_{B^{m}_{2,2}}^{2}\lesssim_{\nu,\alpha}
  \frac{1}{\nu} \norm{ \nabla \theta^\epsilon}_{L^\infty}^{\alpha} \norm{
 \theta^\epsilon}_{L^\infty}^{2-\alpha}\norm{\theta^{\epsilon}}_{B^{m}_{2,2}}^2
 + \norm{ \theta^{\epsilon}}_{L^{2}}^2\norm{\theta^{\epsilon}}_{B^{m}_{2,2}}.
\end{equation}
Indeed, for every $q\in \mathbb{N}$, applying dyadic operator $\Delta_{q}$ to both sides of
regularized equation \eqref{ReMQG2} yields
\begin{equation*}
 \partial_{t}\Delta_{q}\theta^{\epsilon}+\mathcal{J}_{\epsilon}\big(( S_{q+1}u^{\epsilon})
 \cdot\nabla
  \Delta_{q}\theta^{\epsilon} \big)+ \nu   |D|^{\alpha}\Delta_{q}\theta^{\epsilon} =
 \mathcal{J}_{\epsilon}\big( F_{q}( u^{\epsilon},
  \theta^{\epsilon})\big),
\end{equation*}
where
\begin{equation*}
 F_{q}( u^{\epsilon}, \theta^{\epsilon})=(S_{q+1} u^{\epsilon})\cdot\nabla
 \Delta_{q} \theta^{\epsilon}- \Delta_{q}( u^{\epsilon}\cdot\nabla \theta^{\epsilon}).
\end{equation*}
Taking the $L^{2}$ inner product in the above equality with
$\Delta_{q}\theta^{\epsilon}$ and using the divergence free
property, we have
\begin{equation*}
\begin{split}
 \frac{1}{2}\frac{d}{dt}\norm{\Delta_{q}
 \theta^{\epsilon}}_{L^{2}}^{2}+\nu
 \norm{|D|^{\frac{\alpha}{2}}\Delta_{q}\theta^{\epsilon}}_{L^{2}}^{2} & \leq
 \Big|\int_{\mathbb{R}^2} \big(F_q(  u^\epsilon, \theta^\epsilon)\big)(x)\mathcal{J}_\epsilon
 \Delta_q\theta^\epsilon(x)\mathrm{d}x \Big| \\ & \leq
 2^{-q\frac{\alpha}{2}}\norm{F_{q}( u^{\epsilon},
  \theta^{\epsilon})}_{L^2} 2^{q\frac{\alpha}{2}}\norm{\mathcal{J}_{\epsilon}\Delta_{q}\theta^{\epsilon}}_{L^{2}} \\
 & \lesssim 2^{-q\frac{\alpha}{2}}\norm{F_{q}( u^{\epsilon},
  \theta^{\epsilon})}_{L^2}\norm{|D|^{\frac{\alpha}{2}} \Delta_{q}\theta^{\epsilon}}_{L^{2}}.
\end{split}
\end{equation*}
Then by virtue of Young inequality, we deduce
\begin{equation}\label{eq HmEs2}
  \frac{1}{2}\frac{d}{dt}\norm{\Delta_{q}
 \theta^{\epsilon}}_{L^{2}}^{2}+\frac{\nu}{2}
 \norm{|D|^{\frac{\alpha}{2}} \Delta_{q}\theta^{\epsilon}}_{L^{2}}^{2}\leq \frac{C_0}{\nu}
 \Big(2^{-q\frac{\alpha}{2}}\norm{F_{q}( u^{\epsilon},
  \theta^{\epsilon})}_{L^2}\Big)^2.
\end{equation}
From the inequality \eqref{eq commutator2} in the appendix, we know that
\begin{equation}\label{eq FqEst}
\begin{split}
 & 2^{-q\frac{\alpha}{2}}\norm{F_{q}( u^{\epsilon}, \theta^{\epsilon})}_{L^{2}}
 \\ \lesssim & \norm{|D|^{1-\frac{\alpha}{2}}
  u^{\epsilon}}_{L^{\infty}}\sum_{q'\geq
 q-4}2^{(q-q')(1-\frac{\alpha}{2})}\norm{\Delta_{q'} \theta^{\epsilon}}_{L^{2}}+
 \norm{|D|^{\frac{\alpha}{2}} \theta^{\epsilon}}_{L^{\infty}}\sum_{|q'-q|\leq
 4}\norm{\Delta_{q'} \theta^{\epsilon}}_{L^{2}}
\end{split}
\end{equation}
Also notice that for some number $K\in\mathbb{N}$
\begin{equation*}
 \begin{split}
   \norm{|D|^{1-\frac{\alpha}{2}} u^{\epsilon}}_{L^{\infty}}+ \norm{|D|^{\frac{\alpha}{2}} \theta^{\epsilon}}_{L^{\infty}} &\lesssim
   \norm{|D|^{1-\frac{\alpha}{2}} |D|^{\alpha-1}\mathcal{R}^{\bot}\theta^{\epsilon}}_{\dot{B}^{0}_{\infty,1}}+
   \norm{|D|^{\frac{\alpha}{2}} \theta^{\epsilon}}_{\dot{B}^{0}_{\infty,1}}
   \\  & \lesssim
   \sum_{k=-\infty}^{K -1}
   2^{k\alpha/2}\norm{\dot\Delta_k \theta^{\epsilon}}_{L^{\infty}}
   + \sum_{k=
   K }^{\infty}2^{-k(1-\frac{\alpha}{2})}\norm{\dot\Delta_k \nabla\theta^{\epsilon}}_{L^{\infty}}
   \\
   &\lesssim 2^{K\alpha/2} \norm{ \theta^{\epsilon}}_{L^{\infty}}+ 2^{K(\frac{\alpha}{2}-1)}
   \norm{ \nabla\theta^{\epsilon}}_{L^{\infty}},
 \end{split}
\end{equation*}
thus choosing $K$ satisfying $\norm{\theta^\epsilon}_{L^\infty}2^{K}\thickapprox \norm{ \nabla \theta^\epsilon}_{L^\infty}$, we deduce
\begin{equation}\label{eq simplInq}
 \norm{|D|^{1-\frac{\alpha}{2}} u^{\epsilon}}_{L^{\infty}}+ \norm{|D|^{\frac{\alpha}{2}} \theta^{\epsilon}}_{L^{\infty}}
 \lesssim \norm{  \nabla \theta^\epsilon}_{L^\infty}^{\frac{\alpha}{2}}\norm{ \theta^\epsilon}_{L^\infty}^{1-\frac{\alpha}{2}}.
\end{equation}
Plunging the above two estimates \eqref{eq simplInq} and \eqref{eq FqEst} into inequality \eqref{eq HmEs2},
then multiplying both sides by $2^{2qm}$ and summing up over $q\in\mathbb{N}$, we obtain
\begin{equation}\label{eq hmEs-HF}
 \frac{1}{2}\frac{d}{dt}\sum_{q\in\mathbb{N}}2^{2qm}\norm{\Delta_{q}
 \theta^{\epsilon}}_{L^{2}}^{2}+\frac{\nu}{2}\sum_{q\in\mathbb{N}}2^{2qm}
 \norm{|D|^{\frac{\alpha}{2}} \Delta_{q}\theta^{\epsilon}}_{L^{2}}^{2}
 \lesssim \frac{1}{\nu} \norm{  \nabla \theta^\epsilon}_{L^\infty}^{\alpha} \norm{
 \theta^\epsilon}_{L^\infty}^{2-\alpha} \norm{  \theta^\epsilon}_{B^m_{2,2}}^2.
\end{equation}
On the other hand, we apply the low frequency operator $\Delta_{-1}$ to the regularized system \eqref{ReMQG} to get
\begin{equation*}
 \partial_t\Delta_{-1}\theta^\epsilon + \nu  |D|^{\alpha}\Delta_{-1}\theta^{\epsilon}=-\mathcal{J}_\epsilon
 \Delta_{-1}\big(  u^\epsilon \cdot\nabla  \theta^\epsilon \big).
\end{equation*}
Multiplying both sides by $\Delta_{-1}\theta^\epsilon$ and integrating in the spatial variable, we obtain
\begin{equation*}
\begin{split}
 \frac{1}{2}\frac{d}{dt}\norm{\Delta_{-1}\theta^\epsilon}_{L^2}^2+\nu \norm{|D|^{\frac{\alpha}{2}}\Delta_{-1} \theta^\epsilon}_{L^2}^2
 & \leq \Big|\int_{\mathbb{R}^2} \mathrm{div}\Delta_{-1}\big(   u^\epsilon \, \theta^\epsilon\big)(x) \,
 \Delta_{-1}\mathcal{J}_\epsilon \theta^\epsilon(x)\mathrm{d}x\Big| \\
 & \lesssim \norm{  u^\epsilon}_{L^\infty}\norm{  \theta^\epsilon}_{L^2}^2.
\end{split}
\end{equation*}
We see that
\begin{equation}\label{eq fact1}
\begin{split}
 \norm{  u^\epsilon}_{L^\infty} & \leq \Big(\sum_{j\leq -1}+ \sum_{j\geq 0}\Big) \norm{\dot \Delta_j
 |D|^{\alpha-1}\mathcal{R}^{\bot} \theta^\epsilon}_{L^\infty} \\
 & \lesssim\sum_{j\leq -1}2^{j\alpha}\norm{\dot\Delta_j   \theta^\epsilon}_{L^2}+ \sum_{j\geq 0} 2^{j(\alpha-2)} \norm{\dot \Delta_j
 \nabla \theta^\epsilon}_{L^\infty} \\
 & \lesssim \norm{ \theta^\epsilon}_{L^2}+  \norm{\nabla  \theta^\epsilon}_{L^\infty},
\end{split}
\end{equation}
thus we have
\begin{equation}\label{eq hmEs-LF}
 \frac{1}{2}\frac{d}{dt}\norm{\Delta_{-1}\theta^\epsilon}_{L^2}^2+ \frac{\nu}{2} \norm{|D|^{\frac{\alpha}{2}}\Delta_{-1} \theta^\epsilon}_{L^2}^2
 \lesssim \norm{ \theta^\epsilon}_{B^m_{2,2}}\norm{  \theta^\epsilon}_{L^2}^2 .
\end{equation}
Multiplying \eqref{eq hmEs-LF} by $2^{-2m}$ and combining it with \eqref{eq hmEs-HF} leads to \eqref{eq HmEst1}.

Next, we prove that the solution family $(\theta^{\epsilon})$ is
uniformly bounded in $H^{m}$. Indeed, from estimate
\eqref{eq HmEst1}, Besov embedding and the fact that
$\norm{\cdot}^2_{B^m_{2,2}}/C_0\leq \norm{\cdot}^2_{H^m}\leq C_0 \norm{\cdot}^2_{B^m_{2,2}}$ with $C_0$ a universal number, we have
\begin{equation}\label{eq Hm-fact}
\begin{split}
  \frac{d}{dt}\Bigl(\norm{\theta^{\epsilon}(t)}^2_{H^{m}} +\int_0^t \|\theta^\epsilon(\tau)\|^2_{H^{m+\frac{\alpha}{2}}}\mathrm{d}\tau\Bigr)
  & \leq C\Big( \norm{  \nabla \theta^\epsilon}_{L^\infty}^{\alpha}
  \norm{  \theta^\epsilon}_{L^\infty}^{2-\alpha}\norm{\theta^{\epsilon}}_{H^{m}}
   + \norm{ \theta^{\epsilon}}_{L^{2}}\Big) \norm{\theta^{\epsilon}}_{H^{m}}^2
   \\ & \leq   C_1 (1+\norm{\theta^{\epsilon}(t)}_{H^{m}}^2)\norm{\theta^{\epsilon}(t)}_{H^{m}}^2,
\end{split}
\end{equation}
where $C_1$ depends only on $m,\alpha,\nu$.
Gronwall inequality yields that
\begin{equation}\label{HmUnifBounds}
    \sup_{0\leq t\leq T} \norm{\theta^{\epsilon}}_{H^{m}}^2 +\|\theta^\epsilon\|_{L^2_T H^{m+\frac{\alpha}{2}}}^2  \leq
   \frac{\norm{\theta_0}_{H^m}^2}{(\|\theta_0\|_{H^m}^2+1)e^{-CT}-\|\theta_0\|_{H^m}^2}.
\end{equation}
Thus for some
\begin{equation*}
T<\frac{1}{C}\log (1+ 1/\|\theta_0\|_{H^m}^2),
\end{equation*}
the family $(\theta^{\epsilon})$ is uniformly bounded in $\mathcal{C}([0,T],H^{m})\cap L^2([0,T]; H^{m+\frac{\alpha}{2}})$, $m>2$.

\textit{Step 2: Strong Convergence}

We firstly claim that the solutions $(\theta^{\epsilon})$ to the
approximate equation \eqref{ReMQG2} converge in
$\mathcal{C}([0,T],L^{2}(\mathbb{R}^{2}))$. 
Indeed for all $0<\tilde{\epsilon}<\epsilon$, we assume that $\theta^\epsilon$ and $\theta^{\tilde{\epsilon}}$ are two approximate solutions,
then from a direct calculation
\begin{equation*}
 \begin{split}
  \big( \theta^{\epsilon}_{t}-\theta^{ \tilde{\epsilon}}_{t},
  \theta^{\epsilon}-\theta^{\tilde{\epsilon}}\big) =
  -\nu \big(  |D|^{\alpha}\theta^{\epsilon}
  - |D|^{\alpha}\theta^{\tilde{\epsilon}}, \theta^{\epsilon}-\theta^{\tilde{\epsilon}}
  \big)  -\Big( \big( \mathcal{J}_{\epsilon} (  u^{\epsilon} \cdot\nabla
   \theta^{\epsilon}   )-\mathcal{J}_{\tilde{\epsilon}} ( u^{ \tilde{\epsilon}} \cdot\nabla
   \theta^{ \tilde{\epsilon}} )\big),\theta^{\epsilon}-\theta^{ \tilde{\epsilon}}
  \Big),
 \end{split}
\end{equation*}
we have
\begin{equation*}
\begin{split}
 & \frac{1}{2}\frac{d}{dt}\norm{\theta^\epsilon (t)-\theta^{\tilde{\epsilon}}(t)}_{L^2}^2+ \nu
 \norm{ |D|^{\frac{\alpha}{2}}(\theta^\epsilon-\theta^{ \tilde{\epsilon}})}_{L^{2}}^2 \\
  =& \Big( (\mathcal{J}_{\epsilon}-\mathcal{J}_{\tilde{\epsilon}})
    \big(  u^{\epsilon} \cdot\nabla \theta^{\epsilon}  \big),
    \theta^{\epsilon}-\theta^{\tilde{\epsilon}} \Big)
    + \Big( \mathcal{J}_{\tilde{\epsilon}}
    \big(  (u^{\epsilon}-u^{\tilde{\epsilon}})
    \cdot\nabla \theta^{\epsilon}  \big),
    \theta^{\epsilon}-\theta^{\tilde{\epsilon}} \Big)  \\&
     + \Big( \mathcal{J}_{\tilde{\epsilon}}
    \big(  u^{\tilde{\epsilon}} \cdot\nabla
    (\theta^{\epsilon}-\theta^{\tilde{\epsilon}}) \big), \theta^{\epsilon}-\theta^{\tilde{\epsilon}}
    \Big) \\
    := & II_{1}+ II_{2}+ II_{3}.
  \end{split}
\end{equation*}
We set $\delta_0:=\min\{m-\alpha,1 \}$, then for $II_1$, by means of the calculus inequality \eqref{CalInqSob1}, divergence
free condition and the following simple inequality
\begin{equation*}
  \norm{u^{\epsilon}}_{H^{m-\alpha+1}} =\norm{|D|^{\alpha-1}R^{\bot}\theta^{\epsilon}}_{H^{m-\alpha+1}} \lesssim
  \norm{\theta^{\epsilon}}_{H^{m}}\lesssim M,
\end{equation*}
we have
\begin{equation*}
 \begin{split}
   |II_{1}| & \lesssim \epsilon^{\delta_0} \norm{  u^{\epsilon}
    \theta^{\epsilon} }_{H^{1+\delta_0}}
   \norm{\theta^{\epsilon}-\theta^{\tilde{\epsilon}}}_{L^{2}} \\
  & \lesssim  \epsilon^{\delta_0} \big(\norm{u^{\epsilon}}_{H^{1+\delta_0}}
   +  \norm{\theta^{\epsilon}}_{H^{1+\delta_0}}\big)
    \norm{\theta^{\epsilon}-\theta^{\tilde{\epsilon}}}_{L^{2}} \\
   & \lesssim_{M}  \epsilon^{\delta_0}
   \norm{\theta^{\epsilon}-\theta^{\tilde{\epsilon}}}_{L^{2}}.
 \end{split}
\end{equation*}
For $II_2$, we directly obtain
\begin{equation*}
 \begin{split}
  |II_{2}| &  \leq  \norm{ (u^{\epsilon}-u^{\tilde{\epsilon}})\cdot\nabla
   \theta^{\epsilon} }_{\dot H^{-\frac{\alpha}{2}}}\norm{ |D|^{\frac{\alpha}{2}}
  (\theta^{\epsilon}-\theta^{\tilde{\epsilon}})}_{L^{2}} \\ & \leq  C_\alpha \norm
  {|D|^{\alpha-1}\mathcal{R}^{\bot}(\theta^{\epsilon}-\theta^{\tilde{\epsilon}})}_{\dot H^{1-\alpha}}^2
  \norm{\nabla\theta^{\epsilon}}_{\dot H^{\frac{\alpha}{2}}}^2
  + \frac{\nu}{2}\norm{ |D|^{\frac{\alpha}{2}}(\theta^{\epsilon}-\theta^{\tilde{\epsilon}})}_{L^{2}}^2 \\
   &\leq C_{M,\alpha}
  \norm{\theta^{\epsilon}-\theta^{\tilde{\epsilon}}}_{L^{2}}^{2}
  + \frac{\nu}{2}\norm{ |D|^{\frac{\alpha}{2}}(\theta^{\epsilon}-\theta^{\tilde{\epsilon}})}_{L^{2}}^2,
 \end{split}
\end{equation*}
where in the second line we have used the classical product estimate (cf. \cite{ref Dong2}) that for every $s,t<1$ and $s+t>0$,
\begin{equation*}
 \norm{fg}_{\dot H^{s+t-1}}\lesssim_{s,t} \norm{f}_{\dot H^s}\norm{g}_{\dot H^t}.
\end{equation*}
For the last term, $II_{3}$, from the divergence free fact of
$u^{\tilde{\epsilon}}$ and $\mathcal{J}_{\tilde{\epsilon}}\theta^\epsilon =\theta^\epsilon$ we get
\begin{equation*}
  \begin{split}
   II_{3} =   \Big(
    \big(  u^{\tilde{\epsilon}} \cdot\nabla
    (\theta^{\epsilon}-\theta^{\tilde{\epsilon}}) \big),\mathcal{J}_{\tilde{\epsilon}}(
    \theta^{\epsilon}-\theta^{\tilde{\epsilon}})
    \Big)     =   \frac{1}{2} \Big(
      u^{\tilde{\epsilon}},\nabla
    (\theta^{\epsilon}-\theta^{\tilde{\epsilon}})^{2}
    \Big) = 0
  \end{split}
\end{equation*}
Putting all these estimates together yields that for $\delta_0=\min\{m-\alpha,1\}$
\begin{equation*}
  \frac{1}{2}\frac{d}{dt}\norm{\theta^{\epsilon}-\theta^{\tilde{\epsilon}}}_{L^{2}}^{2}\lesssim_{M}
  \big( \epsilon^{\delta_0} + \norm{\theta^{\epsilon}-\theta^{\tilde{\epsilon}}}_{L^{2}}
    \big)
    \norm{\theta^{\epsilon}-\theta^{\tilde{\epsilon}}}_{L^{2}}.
\end{equation*}
Furthermore
\begin{equation*}
  \frac{d}{dt}\norm{\theta^{\epsilon}-\theta^{\tilde{\epsilon}}}_{L^{2}}\leq
  C(M) \big( \epsilon^{\delta_0} + \norm{\theta^{\epsilon}-\theta^{\tilde{\epsilon}}}_{L^{2}}
    \big).
\end{equation*}
Thus the Gr\"onwall inequality leads to the desired result:
\begin{equation}\label{L2contra}
 \begin{split}
   \sup_{0\leq t\leq T}
   \norm{\theta^{\epsilon}-\theta^{\tilde{\epsilon}}}_{L^{2}}& \leq
   \; e^{C(M)T}\big(\epsilon^{\delta_0}
   +\norm{\theta^{\epsilon}_{0}-\theta^{\tilde{\epsilon}}_{0}}_{L^{2}}\big) \\
   & \lesssim_{T,\norm{\theta_0}_{H^m}} a(\epsilon),
 \end{split}
\end{equation}
where $a(\epsilon):= \epsilon^{\delta_0} +\norm{(Id-\mathcal{J}_\epsilon) \theta_0}_{L^2}$ satisfies that $a(\epsilon)\rightarrow 0$ as $\epsilon\rightarrow 0$.

From \eqref{L2contra}, we deduce that the solution family
$(\theta^{\epsilon})$ is Cauchy sequence in
$\mathcal{C}([0,T],L^{2}(\mathbb{R}^{2}))$, so that it converges strongly to a
function $\theta\in \mathcal{C}([0,T],L^{2}(\mathbb{R}^{2}))$. This result
combined with uniform bounds \eqref{HmUnifBounds} and the
interpolation inequality in Sobolev spaces gives that for all $0\leq
s <m$
\begin{equation*}
 \begin{split}
  \sup_{0\leq t\leq T}\norm{\theta^{\epsilon}-\theta}_{H^{s}}& \leq
  C_{s}\sup_{0\leq t\leq
  T}(\norm{\theta^{\epsilon}-\theta}_{L^{2}}^{1-s/m}\norm{\theta^{\epsilon}-\theta}_{H^{m}}^{s/m})\\
   & \lesssim_{s,T,\norm{\theta_{0}}_{H^{m}}} a(\epsilon) ^{1-s/m}.
 \end{split}
\end{equation*}
Hence we obtain the strong convergence in
$\mathcal{C}([0,T],H^{s}(\mathbb{R}^{2}))$ for all $ s<m$. With $2<s<m$, this
specially implies strong convergence in
$\mathcal{C}([0,T],\mathcal{C}^{1}(\mathbb{R}^{2}))$. Also from the equation
\begin{equation*}
   \theta^{\epsilon}_{t}=-\nu |D|^{\alpha}\theta^{\epsilon}
  -\mathcal{J}_{\epsilon}( u^{\epsilon} \cdot\nabla \theta^{\epsilon}),
\end{equation*}
we find that $\theta^{\epsilon}_{t} $ strongly converges to $-\nu
|D|^{\alpha}\theta - u\cdot\nabla\theta$ in
$\mathcal{C}([0,T],L^2(\mathbb{R}^{2}) )$. Since $\theta^{\epsilon}\rightarrow
\theta$, the distribution limit of $\theta^{\epsilon}_{t}$ has to be
$\theta_{t}$. Thus $\theta \in \mathcal{C}^{1}([0,T],L^2(\mathbb{R}^{2}))\cap
\mathcal{C}([0,T],\mathcal{C}^{1}(\mathbb{R}^{2}))$ is a solution to the
original equation \eqref{MQG}. Using Fatou's Lemma, from \eqref{HmUnifBounds}, we also have $\theta\in
L^{\infty}([0,T],H^{m}(\mathbb{R}^{2}))\cap L^2([0,T],H^{m+\frac{\alpha}{2}}(\mathbb{R}^2))$.

Next, we show that $\theta\in \mathcal{C}([0,T],H^{m}(\mathbb{R}^2))$ indeed. The proof is classical (cf. \cite{KisNS}).
We first prove that $\theta(t)\rightarrow \theta_0$ weakly in $H^m$ as $t\rightarrow 0$. Let $\psi(x)\in\mathcal{C}^\infty_0 (\mathbb{R}^2)$,
denote
$$F^\epsilon(t,\psi):= (\theta^\epsilon, \psi)=\int_{\mathbb{R}^2}\theta^\epsilon(t,x)\psi(x)\mathrm{d}x.$$
Clearly $F^\epsilon(\cdot,\psi)\in \mathcal{C}([0,T])$. And by taking the inner product of \eqref{ReMQG2} with $\psi$, we get
\begin{equation*}
  \frac{d}{dt}F^\epsilon(t,\psi) = -(u^\epsilon\theta^\epsilon, \mathcal{J}_\epsilon \nabla\psi ) - \nu(\theta^\epsilon, |D|^\alpha \psi),
\end{equation*}
thus for every $p\in ]1,2]$
\begin{equation*}
  \int_0^T |F^\epsilon_t|^p\mathrm{d} t \leq T^{\frac{1}{p}-\frac{1}{2}} \| u^\epsilon\|_{L^2_T L^2}\|\theta^\epsilon \|_{L^{\infty}_T L^2} \|\psi\|_{H^3}+\nu \|\theta^\epsilon\|_{L^p_T L^2}\|\psi\|_{H^\alpha}.
\end{equation*}
From the $L^2$ energy estimate $\|\theta^\epsilon\|_{L^\infty_T L^2}^2+\|\theta^\epsilon\|_{L^2_T \dot H^{\alpha/2}}^2 \leq \|\theta_0\|_{L^2}^2 $, we know $\|F^\epsilon_t(\cdot,\psi)\|_{L^p([0,T])}\lesssim_{T,\|\psi\|_{H^3}}1$. Hence by Arzela-Ascolli theorem, $\{F^\epsilon(t,\psi)\}_{\epsilon>0}$ is compact in $\mathcal{C}([0,T])$, and we can choose a subsequence $F^{\epsilon_j}(t,\psi)$ converging to a function $F(t,\psi)\in \mathcal{C}([0,T])$ uniformly in $t$. In particular, from $\theta^\epsilon\rightarrow \theta$ in $\mathcal{C}([0,T]; L^2)$, we can further find a subsequence (still denote $F^{\epsilon_j}$) such that $F(t,\psi)= (\theta(t), \psi)$ for all $t\in [0,T]$. Next, since $\mathcal{C}_0^\infty(\mathbb{R}^2)$ is dense in $H^{-m}(\mathbb{R}^2)$ and
$\|\theta^\epsilon(t)\|_{H^m}$ is uniformly bounded in $[0,T]$, $F^{\epsilon_j}(t,\psi)$ converges to $F(t,\psi)$ for every $\psi\in H^{-m}$.
Then for every $t>0$ and $\psi\in H^{-m}$
\begin{equation*}
  |(\theta(t)-\theta_0,\psi)|\leq |(\theta(t)-\theta^{\epsilon_j}(t),\psi)|+ |(\theta^{\epsilon_j}(t)-\theta^{\epsilon_j}_0 ,\psi)|+ |(\theta_0^{\epsilon_j}-\theta_0,\psi)|.
\end{equation*}
All the three terms in the RHS can be made small for sufficiently small $\epsilon_j$ and $t$, thus $\theta(t)$ converges to $\theta_0$ weakly in $H^m$ as $t\rightarrow 0$. So we have
\begin{equation}\label{eq CntLinEs}
  \|\theta_0\|_{H^m}\leq \liminf_{t\rightarrow 0}\|\theta(t)\|_{H^m}.
\end{equation}
Furthermore, from \eqref{eq Hm-fact} we infer that for every $\epsilon>0$ the function $\|\theta^\epsilon(t)\|_{H^m}^2$ is below the graph of the solution of the equation
\begin{equation*}
  \frac{d}{dt}y(t)= C y(t)+ C y^2(t), \quad y(0)=\|\theta_0 \|_{H^m}^2.
\end{equation*}
By construction, the same holds for $\|\theta(t)\|_{H^m}^2$. Thus from the continuity of $y(t)$, we find $\|\theta_0\|_{H^m}\geq \limsup_{t\rightarrow 0}\|\theta(t)\|_{H^m}$. Therefore $\|\theta_0\|_{H^m}=\lim_{t\rightarrow 0}\|\theta(t)\|_{H^m}$, and the conclusion follows from this fact combined with the weak convergence.

\textit{Step 3: Uniqueness}

Let $\theta^{1}$, $\theta^{2}\in L^\infty ([0,T], H^{m}(\mathbb{R}^2))$ be two smooth solutions to the
modified quasi-geostrophic equation \eqref{MQG} with the same
initial data. Denote $u^{i}=|D|^{\alpha-1}R^{\bot}\theta^{i}$,
$i=1,2$, $\delta \theta=\theta^{1}-\theta^{2} $, $\delta u=
u^{1}-u^{2}$, then we write the difference equation as
\begin{equation*}
  \partial_{t}\delta\theta + u^{1}\cdot\nabla\delta\theta + \nu
  |D|^{\alpha}\delta\theta = -\delta u\cdot\nabla \theta^{2},\quad
  \delta\theta|_{t=0}=0
\end{equation*}
We also use the $L^{2}$ energy method, and in a similar way as
treating the term $II_{3}$, we obtain
\begin{equation*}
  \frac{d}{dt} \norm{\delta\theta}_{L^{2}}\leq C_{\alpha}
  \norm{\nabla\theta^{2}}_{\dot H^{\frac{\alpha}{2}}}^2 \norm{\delta\theta}_{L^{2}} \leq C_{\alpha}
  \norm{\theta^{2}}_{H^{m}}^2\norm{\delta\theta}_{L^{2}}.
\end{equation*}
Thus the Gr\"onwall inequality ensures $\delta\theta\equiv0$, that
is, $\theta^{1}\equiv\theta^{2}$.

\textit{Step 4: Smoothing Effect}

Precisely, we have that for all $\gamma\in \mathbb{R}^{+}$ and $t\in [0,T]$
\begin{equation}\label{eq SmthEst}
 \norm{t^{\gamma} \theta(t)} _{L^{\infty}_{T} H^{m+\gamma
 \alpha}}^2 +\|t^\gamma \theta(t)\|_{L^2_T H^{m+\alpha/2+\gamma\alpha}}^2 \leq C e^{C (\gamma+1)
 (T\norm{\theta}_{L^{\infty}_{T}H^{m}}^2+T)} \norm{\theta_0}_{ H^{m}}^2,
\end{equation}
where $C$ is an absolute constant depending only on $\alpha,\nu,m$.
Notice that $t^{\gamma}\theta$ ($\gamma >0$) satisfies
\begin{equation}\label{eq SmtEq}
  \partial_{t}(t^{\gamma}\theta) + u\cdot\nabla (t^{\gamma}\theta)
  +\nu |D|^{\alpha}(t^{\gamma}\theta)= \gamma t^{\gamma-1}\theta,
  \quad (t^{\gamma}\theta)|_{t=0} =0.
\end{equation}
which is a linear transport-diffusion equation with the
velocity $u=|D|^{\alpha-1}R^{\bot}\theta$, $\alpha\in ]0,2[$. We first treat the case
$\gamma\in \mathbb{Z}^{+}$. For $\gamma=1$, in a similar way as obtaining \eqref{eq HmEst1}, and using the Sobolev embedding
we infer
\begin{equation*}
\begin{split}
    \frac{d}{dt}\|t\theta(t)\|_{B^{m+\alpha}_{2,2}}^2 + \| t\theta(t)\|_{B^{m+\frac{3}{2}\alpha}_{2,2}}^2 & \lesssim (\|\nabla\theta(t)\|_{L^\infty}^\alpha\|\theta(t)\|_{L^\infty}^{2-\alpha} +\|\theta(t)\|_{L^2}) \|t\theta(t)\|_{B^{m+\alpha}_{2,2}}^2 + \|\theta(t)\|_{B^{m+\frac{\alpha}{2}}_{2,2}}^2
\\ & \lesssim (\|\theta(t)\|_{H^m}^2 +1) \|t\theta(t)\|_{B^{m+\alpha}_{2,2}}^2 + \|\theta(t)\|_{B^{m+\frac{\alpha}{2}}_{2,2}}^2.
\end{split}
\end{equation*}
Gronwall inequality yields that
\begin{equation}\label{eq SmthEst2}
\begin{split}
  \|t\theta(t)\|_{B^{m+\alpha}_{2,2}}^2 + \|t\theta(t)\|^2_{L^2_T B^{m+\frac{3}{2}\alpha}_{2,2}} \lesssim e^{CT+C T \| \theta\|_{L^\infty_T H^m}^2} \int_0^T\|\theta(\tau)\|_{B^{m+\frac{\alpha}{2}}_{2,2}}^2\mathrm{d}\tau.
\end{split}
\end{equation}
Meanwhile, similarly as obtaining \eqref{eq Hm-fact}, we get
\begin{equation}\label{eq SmthEst3}
  \|\theta(t)\|_{H^m}^2 +\|\theta\|_{L^2_T H^{m+\frac{\alpha}{2}}}^2 \leq \|\theta_0\|_{H^m}^2 e^{CT+ CT\|\theta\|_{L^\infty_T H^m}^2}.
\end{equation}
Thus \eqref{eq SmthEst} with $\gamma=1$ follows from \eqref{eq SmthEst2} and \eqref{eq SmthEst3} and the fact that the space $B^s_{2,2}$ is equivalent with $H^s$, $s\in\mathbb{R}$. Now suppose estimate \eqref{eq SmthEst} holds for $\gamma=N$, we shall
consider the case $N+1$. We use the equation \eqref{eq SmtEq} with $\gamma=N+1$. Similarly as above, and observing that the constant $C$ in \eqref{eq SmthEst2} is independent of $N$ if $\theta(t)$ is replaced by $t^N \theta(t)$ and $m$ by $m+N\alpha$, we have
\begin{equation*}
\begin{split}
  \|t^{N+1}\theta(t)\|_{H^{m+(N+1)\alpha}}^2 + \|t^{N+1}\theta(t)\|_{L^2_T H^{m+(N+1)\alpha+\frac{\alpha}{2}}}^2  & \lesssim e^{CT+ CT\|\theta\|_{L^\infty_T H^m}^2} \|t^{N}\theta(t)\|_{L^2_T H^{m+(N+\frac{1}{2})\alpha}}^2 \\
  &\lesssim e^{C(N+2)(T+ T\|\theta\|_{L^\infty_T H^m}^2)}\|\theta_0\|_{H^m}^2.
\end{split}
\end{equation*}
Thus the induction method ensures the estimate \eqref{eq SmthEst}
for all $\gamma\in \mathbb{Z}^{+}$. Also notice that for $\gamma=0$
the inequality \eqref{eq SmthEst} is also satisfied. Hence we
obtain estimate \eqref{eq SmthEst} for all $\gamma\in
\mathbb{N}$. For the general $\gamma \geq 0 $, we set
$[\gamma]\leq \gamma< [\gamma]+1$, where $[\gamma]$ denotes the
integer part of $\gamma$, and use the interpolation inequality in
Sobolev spaces to get
\begin{equation*}
 \begin{split}
   \norm{t^{\gamma} \theta}^2_{L^{\infty}_{T} H^{m+\gamma \alpha}}
   \leq & \|t^{[\gamma]} \theta\|_{L^{\infty}_{T} H^{m+[\gamma]
   \alpha}}^{2([\gamma]+1-\gamma)}\|t^{[\gamma]+1} \theta\|_{L^{\infty}_{T}
   H^{m+([\gamma]+1) \alpha}}^{2(\gamma-[\gamma])}
    \\  \lesssim &
    e^{C (\gamma+1) (T+ T\norm{\theta}_{L^{\infty}_{T}H^{m}}^2)}
   \norm{\theta_{0}}_{H^{m}}^2.
 \end{split}
\end{equation*}
Similar estimate holds for $\|t^\gamma\theta\|^2_{L^2_T H^{m+(\gamma+\frac{1}{2})\alpha}}$.

Therefore,  we conclude the Proposition \ref{prop local}.
\end{proof}

Now, we are devoted to building the blowup criterion.

\begin{proof}[Proof of Proposition \ref{prop break}]
We first note that the equation has a natural blowup criterion: if
$T^{*}<\infty$ then necessarily
\begin{equation*}
  \norm{\theta}_{L^{\infty}([0,T^{*}),H^{m})} + \norm{\theta}_{L^2([0,T^*), H^{m+\frac{\alpha}{2}})}=\infty.
\end{equation*}
Otherwise from the local result, the solution will continue over
$T^{*}$.

In the same way as obtaining the estimate \eqref{eq HmEst1}, we
get the similar result for the original equation
\begin{equation*}\label{eq HmGdEst}
  \frac{1}{2}\frac{d}{dt}\norm{\theta(t)}_{B^{m}_{2,2}}^{2}+
  \frac{\nu}{2}\norm{\theta(t)}_{B^{m+\frac{\alpha}{2}}_{2,2}}^{2}\leq
  C_{m,\alpha}\Big( \frac{1}{\nu} \norm{ \nabla \theta}_{L^\infty}^{\alpha} \norm{
 \theta}_{L^\infty}^{2-\alpha}\norm{\theta}_{B^{m}_{2,2}}^2
 + \norm{\theta}_{L^{2}}^2\norm{\theta}_{B^{m}_{2,2}}\Big).
\end{equation*}
Also due to the maximum principle Proposition \ref{prop MP}, we have
\begin{equation*}\label{eq HmGdEst1}
  \frac{d}{dt}\Big(\norm{\theta(t)}_{B^{m}_{2,2}}^2 + \nu \int_0^t \norm{\theta(\tau)}_{B^{m+\frac{\alpha}{2}}_{2,2}}^2\mathrm{d} \tau  \Big)\lesssim_{\alpha,\nu,m}
  \big(\norm{\nabla
  \theta(t)}_{L^{\infty}}^{\alpha}+ 1 \big)\norm{\theta(t)}_{B^{m}_{2,2}}^2.
\end{equation*}
This together with the Gr\"onwall inequality leads to
\begin{equation*}
\begin{split}
  \sup_{0\leq t\leq T}\norm{\theta(t)}_{H^{m}}^2 + \norm{\theta}_{L^2([0,T],H^{m+\frac{\alpha}{2}})}^2& \leq C_0\sup_{0\leq t\leq T} \norm{\theta(t)}_{B^m_{2,2}}^2 + C_0 \norm{\theta}_{L^2([0,T],B^{m+\frac{\alpha}{2}}_{2,2})}^2 \\ &
  \leq C  \exp\Big\{CT+  C\int_{0}^{T} \norm{\nabla\theta(t)}_{L^{\infty}}^\alpha\textrm{d} t \Big\}.
\end{split}
\end{equation*}
Further, if $T^{*}<\infty$ and the integral $\int_{0}^{T^{*}}
\norm{\nabla\theta(t)}_{L^{\infty}}^\alpha\textrm{d} t <\infty$, then from
the above estimate we directly have
\begin{equation*}
 \sup_{0\leq t< T^{*}}\norm{\theta(t)}_{H^{m}} + \norm{\theta}_{L^2([0,T^*), H^{m+\frac{\alpha}{2}})}<\infty.
\end{equation*}
Clearly this contradicts
the upper natural blowup criterion. Thus, if $T^{*}<\infty$, we necessarily have the equality
$\int_{0}^{T^{*}} \norm{\nabla\theta(t)}_{L^{\infty}}^\alpha\textrm{d} t
=\infty.$

\end{proof}

\section{Global Existence}
\setcounter{section}{5}\setcounter{equation}{0}

In this section, we use the modulus of continuity argument developed by Kiselev, Nazarov
and Volberg \cite{KisNV} to prove the global result. Throughout this
section, we assume $T^{*}$ be the maximal existence time of the
solution in $\mathcal{C}([0,T^{*}),H^{m})\cap L^2([0,T^*), H^{m+\frac{\alpha}{2}})$.

Let $\lambda>0$ be a real number which will be chosen later, then we
define the set
\begin{equation*}
   \mathcal{I}:=\big\{ T\in [0,T^{*})| \forall t\in [0,T], \forall x,y\in \mathbb{R}^{2},x\neq y,
   |\theta(t,x)-\theta(t,y)|< \omega_{\lambda}(|x-y|)\big\},
\end{equation*}
where $\omega$ is a strict modulus of continuity also satisfying that
$\omega'(0)<\infty$, $\lim_{\eta \searrow 0}\omega''(\eta)=-\infty$
and
$$\omega_{\lambda}(|x-y|)=\omega(\lambda|x-y|).$$ The
explicit expression of $\omega$ will be shown later (i.e. \eqref{eq modulus}).

We first show that the set $\mathcal{I}$ is nonempty, that is, at
least $0\in \mathcal{I}$. The proof is almost the same with the one
in \cite{ref AbidiH} only by setting $T_{1}$ there to be $0$. We omit it here and
only note that to fit our purpose $\lambda$ can be taken
\begin{equation}\label{eq lambdValue}
  \lambda=\frac{\omega^{-1}(3\norm{\theta_{0}}_{L^{\infty}})}
  {2\norm{\theta_{0}}_{L^{\infty}}}\norm{\nabla\theta_{0}}_{L^{\infty}}.
\end{equation}

Thus $\mathcal{I}$ is an interval of the form $[0,T_{*})$, where
$T_{*}:=\sup_{T\in\mathcal{I}}T$. We have three possibilities:
\begin{enumerate}[(a)]
  \item\label{case a} $T_{*}=T^{*} $
  \item\label{case b} $T_*< T^*$ and $T_{*}\in\mathcal{I}$
  \item\label{case c} $T_*< T^*$ and $T_{*}\notin \mathcal{I}$
\end{enumerate}

For case (a), we necessarily have $T^{*}=\infty$, since the Lipschitz
norm of $\theta$ does not blow up from the definition of
$\mathcal{I}$ which contradicts with \eqref{eq blowup}. This is our goal.

For case (b), we observe that this is just the case treated in \cite{ref
AbidiH} or \cite{ref Dong} showing that it is impossible. The proof only needs very small
modification, so we omit it either. We just point out in this case
the smoothing effects will be used, since we need the fact that
$\norm{\nabla^{2}\theta(T_{*})}_{L^{\infty}}$ is finite.

Then our task is reduced to get rid of the case (c). We prove by contradiction. If the case (c) is
satisfied, then by the time continuity of $\theta$, we necessarily get
\begin{equation*}
 \sup_{x,y\in\mathbb{R}^2,x\neq y}\frac{|\theta(T_*,x)-\theta(T_*,y)|}{\omega_\lambda(|x-y|)}=1.
\end{equation*}
We further have the following assertion (with its proof in the end of this section).

\begin{lemma}\label{lem fact}
If $T_*<T^*$ is the first time that the strict modulus of continuity $\omega_\lambda$ is lost (i.e. case (c)),
then there exists $x,y\in\mathbb{R}^{2} $, $x\neq y$ such that
\begin{equation}\label{eq fact}
  \theta(T_{*},x)-\theta(T_{*},y)=\omega_{\lambda}(\xi), \quad \text{with} \quad \xi:=|x-y|.
\end{equation}
Moreover, let $\ell=\frac{x-y}{|x-y|}$ and $v\in \mathbb{S}^1$ be the unit vector perpendicular to $\ell$, we have
\begin{equation}\label{eq fact0}
  \partial_\ell \theta(T_*,x)=\partial_\ell \theta(T_*,y)=\omega_\lambda'(\xi),\quad \partial_v \theta(T_*,x)=\partial_v \theta(T_*,y)=0,
\end{equation}
where $\partial_\ell =\ell \cdot \nabla$ and $\partial_v=v\cdot\nabla $ are the directional derivatives along $\ell$ and $v$ respectively.
\end{lemma}

We shall show that this scenario \eqref{eq fact} can not happen, more precisely, we shall prove
\begin{equation*}
  f'(T_{*})<0, \quad \textrm{with} \quad
  f(t):=\theta(t,x)-\theta(t,y).
\end{equation*}
This is impossible because we
necessarily have $f(t) \leq f(T_{*})$, for all $0\leq t\leq T_{*}$ from the definition of $\mathcal{I}$.

We see that the modified quasi-geostrophic equation \eqref{MQG} can
be defined in the classical sense (from the smoothing effect), and thus
\begin{equation*}
 \begin{split}
  f'(T_{*}) = &  -\Big[(u\cdot\nabla \theta)(T_{*},x)- (u\cdot\nabla \theta)(T_{*},y)
   \Big] + \nu \Big[(-|D|^{\alpha}\theta)(T_{*},x) - (-|D|^{\alpha}\theta)(T_{*},y)
   \Big] \\ :=& \, \mathcal{A}_{1} + \mathcal{A}_{2}
  \end{split}
\end{equation*}
with
 $$u=|D|^{\alpha-1}\mathcal{R}^{\bot}\theta=\mathcal{R}^{\bot}_{\alpha}\theta :=(-\mathcal{R}_{\alpha,2
 }\theta,\mathcal{R}_{\alpha,1}\theta)$$
 where $\mathcal{R}_{\alpha,j}$
are the modified Riesz transforms introduced in the section \ref{sec MOC}.

For the first term, $\mathcal{A}_{1}$, from \eqref{eq fact}, we find that
\begin{equation*}
\begin{aligned}
  \mathcal{A}_1 & = [(u(T_*,x)-u(T_*,y))\cdot \ell]  \omega_\lambda'(\xi) \\
  & = [(u(T_*,x)-u(T_*,y))\cdot \ell] \lambda  \omega'(\lambda\xi).
\end{aligned}
\end{equation*}
Lemma \ref{lem ModuliModRiez} gives us a rough estimate as follows
\begin{equation*}
  |\mathcal{A}_{1} | \leq \Omega_{\lambda}(\xi)\lambda\omega'(\lambda\xi)=\lambda^{\alpha} (\Omega\omega')(\lambda\xi),
\end{equation*}
where $\Omega_{\lambda}(\xi)$ is defined from \eqref{eq ModuliMRiez}, i.e.
\begin{equation}\label{eq omega}
\Omega_{\lambda}(\xi)=A
\bigg(\int^{\xi}_{0}\frac{\omega_{\lambda}(\eta)}{\eta^{\alpha}}\textrm{d}
\eta+
 \xi\int_{\xi}^{\infty}\frac{\omega_{\lambda}(\eta)}{\eta^{1+\alpha}}\textrm{d}
 \eta\bigg) = \lambda^{\alpha-1}\Omega(\lambda \xi).
\end{equation}

For the second term, $\mathcal{A}_{2}$, from Lemma \ref{lem Modulidiffer} we get
\begin{equation*}
\begin{split}
 \mathcal{A}_{2}\leq \nu \lambda^{\alpha}\Upsilon(\lambda \xi),
\end{split}
\end{equation*}
where
\begin{equation*}
\begin{split}
 \Upsilon(\xi):= &  \, B
 \int_{0}^{\frac{\xi}{2}}\frac{\omega(\xi+2\eta)+\omega (\xi-2\eta)
 -2\omega (\xi)}{\eta^{1+\alpha}}\textrm{d} \eta \\
  & + B\int_{\frac{\xi}{2}}^{\infty}\frac{\omega (2\eta+\xi)-\omega (2\eta-\xi)
 -2\omega (\xi)}{\eta^{1+\alpha}}\textrm{d}\eta.
\end{split}
\end{equation*}
Thus we obtain
\begin{equation}\label{eq f'}
  f'(T_{*})\leq \lambda^{\alpha}\big(\Omega \omega' + \nu \Upsilon)(\lambda
  \xi\big).
\end{equation}

Observe that when $\alpha\in ]1,2[$ and $\xi$ is a large number, the integral from $0$ to $\xi$ in the expression of $\Omega$ always produce much difficulty,
roughly speaking, to ensure the RHS of \eqref{eq f'} is negative, we need that the contribution from this part
$\omega(\xi)\omega'(\xi)\ll \frac{\omega(\xi)}{\xi^\alpha}$, thus it seems impossible to construct an appropriate unbounded MOC.
However, basically following a idea from \cite{Kis},
we can further develop the contribution of the dissipative term and use the additional dissipation to control the "bad" part of the nonlinearity so that
improved estimates of $\mathcal{A}_1$ and $\mathcal{A}_2$ can be obtained. Meanwhile, when $\alpha=1$ this method can also slightly improve the MOC constructed in
\cite{KisNV}. Precisely,
\begin{lemma}\label{lem ImpMdu}
  Under the condition of Lemma \ref{lem fact} and for $\alpha\in [1,2[$, we have
\begin{equation}
  \mathcal{A}_2 \leq \nu\lambda^{\alpha} \Upsilon(\lambda \xi)+ \nu\lambda^\alpha \Upsilon^\bot(\lambda\xi),
\end{equation}
where $\Upsilon^\bot\leq 0$ is a meaningful integral defined from $\theta$ and $\omega$ (with its explicit formula cf. Lemma 5.5 of \cite{MiaoXue2}). Correspondingly,
we can treat the drift term as follows
\begin{equation}
  |(u(T_*,x)-u(T_*,y))\cdot \ell|\leq \lambda^{\alpha-1} \widetilde{\Omega}(\lambda\xi)
\end{equation}
with
\begin{equation}
  \widetilde{\Omega}(\xi)=A \Big(-\xi\Upsilon^\bot(\xi) + \xi \int_\xi^\infty \frac{\omega(\eta)}{\eta^{1+\alpha}}\mathrm{d}\eta + \xi^{-\alpha+1}\omega(\xi) \Big),
\end{equation}
where $A$ is an absolute constant that may depend on $\alpha$.
\end{lemma}

\begin{proof}
  This is a direct consequence of Lemma 5.5 and Lemma 5.6 in \cite{MiaoXue2} when $\alpha\in]1,2[$, and they can simply extend to the case $\alpha=1$.
\end{proof}
Hence when $\alpha\in [1,2[$, based on Lemma \ref{lem ImpMdu}, we also get
\begin{equation}
  f'(T_*)\leq  \lambda^\alpha (\widetilde{\Omega}\omega' +\nu \Upsilon +\nu \Upsilon^\bot )(\lambda\xi).
\end{equation}

Next we shall construct our special modulus of continuity in the
spirit of \cite{KisNV}. Let $0<\gamma<\delta<1 $ be two small positive numbers chosen later, and define the continuous functions $\omega$ as follows
\begin{equation}\label{eq modulus}
\mathrm{MOC}\begin{cases}
 \omega(\xi)=\xi-\xi^{1+\frac{\alpha}{2}}
 \quad & \text{if} \quad 0\leq \xi\leq \delta, \\
 \omega'(\xi)=\frac{  \gamma}{4\xi}
 \quad & \text{if} \quad \xi>\delta,
\end{cases}
\end{equation}
equivalently,
\begin{equation}\label{eq modulus0}
 \omega(\xi)=
 \begin{cases}
 \xi-\xi^{1+\frac{\alpha}{2}}
 \quad & \text{if} \quad 0\leq \xi\leq \delta, \\
 \delta-\delta^{1+\frac{\alpha}{2}}+\frac{\gamma}{4}\log\frac{\xi}{\delta}
 \quad & \text{if} \quad \xi>\delta.
\end{cases}
\end{equation}
Note that, for small $\delta$, the left derivative of $\omega$ at
$\delta$ is about 1, while the right derivative equals
$\frac{\gamma}{4\delta}<\frac{1}{4}$.
So $\omega$ is concave if $\delta$ is small enough. Clearly,
$\omega(0)=0$, $\omega'(0)=1$ and $\lim_{\eta\rightarrow
0+}\omega''(\eta)= -\infty$, and $\omega$ is unbounded (it has the
logarithmic growth at infinity).

Then our target is to show that, for this MOC $\omega$, when $\alpha\in ]0,1[$
\begin{equation}\label{eq target1}
  \Omega(\xi)\omega'(\xi)+ \nu\Upsilon(\xi)<0  \qquad \textrm{for all}\quad\xi>0,
\end{equation}
and when $\alpha\in[1,2[$
\begin{equation}\label{eq target2}
  \widetilde{\Omega}(\xi) \omega'(\xi) + \nu \Upsilon(\xi) +\nu \Upsilon^\bot(\xi)<0 \qquad \textrm{for all}\quad\xi>0.
\end{equation}
In the following we shall carefully check these two formulae.
\vskip0.2cm
\textbf{Case \textrm{I}:} when $\alpha\in]0,1[$

Precisely, we shall check the following inequality
\begin{equation*}
\begin{split}
 A \bigg[\int^{\xi}_{0}\frac{\omega(\eta)}{\eta^{\alpha}}\textrm{d} \eta+&
 \xi\int_{\xi}^{\infty}\frac{\omega(\eta)}{\eta^{1+\alpha}}\textrm{d}
 \eta\bigg]\omega'(\xi)
 +\nu B \int_{0}^{\frac{\xi}{2}}\frac{\omega(\xi+2\eta)+\omega(\xi-2\eta)-2\omega(\xi)}{\eta^{1+\alpha}}\textrm{d}
 \eta \\
 +& \nu B\int_{\frac{\xi}{2}}^{\infty}\frac{\omega(2\eta+\xi)-\omega(2\eta-\xi)-2\omega(\xi)}{\eta^{1+\alpha}}\textrm{d}
 \eta <0 \quad \text{for all}\quad \xi>0.
\end{split}
\end{equation*}
We further divide it into two cases.

\textbf{Case \textrm{I}.1:} $\alpha\in]0,1[$ and $0<\xi\leq \delta$
\vskip0.2cm

Since $\frac{\omega(\eta)}{\eta}\leq \omega'(0)=1$ for all $\eta>0$
and $\eta\leq \eta^{\alpha}$ for $\eta\leq \delta<1 $, we have
$$\int_{0}^{\xi}\frac{\omega(\eta)}{\eta^{\alpha}}\textrm{d} \eta \leq
\int_{0}^{\xi}\frac{\omega(\eta)}{\eta}\textrm{d} \eta \leq \xi,$$
 and
$$\int_{\xi}^{\delta}\frac{\omega(\eta)}{\eta^{1+\alpha}}\textrm{d}
 \eta \leq \int_{\xi}^{\delta}\frac{1}{\eta^{\alpha}}\textrm{d}
 \eta = \frac{1}{1-\alpha} (\delta^{1-\alpha}-\xi^{1-\alpha})\leq \frac{1}{1-\alpha}.$$
  Furthermore,
\begin{equation*}
 \int_{\delta}^{\infty}\frac{\omega(\eta)}{\eta^{1+\alpha}}\textrm{d}
 \eta=\frac{1}{\alpha}\frac{\omega(\delta)}{\delta^{\alpha}}+
 \frac{1}{\alpha}\int_{\delta}^{\infty}\frac{  \gamma}{4\eta^{1+\alpha}}\textrm{d}
 \eta\leq \frac{1}{\alpha}+  \frac{1}{\alpha^2}
 \frac{\gamma}{\delta^{\alpha}}\leq \frac{2}{\alpha},
\end{equation*}
if $\gamma<\alpha \delta$. Clearly $\omega'(\xi)\leq\omega'(0)=1$, so we get that the positive
part is bounded by $ A \xi\frac{2}{\alpha(1-\alpha)}$.

For the negative part, we have
\begin{equation*}
\begin{split}
 \nu B \int_{0}^{\frac{\xi}{2}}&\frac{\omega(\xi+2\eta)+\omega(\xi-2\eta)-2\omega(\xi)}{\eta^{1+\alpha}}\textrm{d}
 \eta \leq
 \nu B\int_{0}^{\frac{\xi}{2}}\frac{\omega''(\xi)2\eta^{2}}{\eta^{1+\alpha}}\textrm{d}
 \eta\\  =  &
 -\nu B\frac{\alpha(2+\alpha)}{2^{1-\alpha}(2-\alpha)}\xi^{1-\frac{\alpha}{2}}\leq
 -\frac{\alpha}{2}\nu B\xi^{1-\frac{\alpha}{2}}.
\end{split}
\end{equation*}
But, clearly $\xi\Big(A \frac{2}{\alpha(1-\alpha)} -\frac{\alpha}{2}\nu B
\xi^{-\frac{\alpha}{2}}\Big)<0$ on
 $]0,\delta]$ when $\delta$ is small enough.
\vskip0.2cm
 \textbf{Case \textrm{I}.2:} $\alpha\in]0,1[$ and $\xi\geq \delta$
\vskip0.2cm

For $\eta\leq\delta<1$ we still use
$\eta^\alpha\geq \eta$ and for $\delta\leq
\eta\leq\xi$ we use $\omega(\eta)\leq\omega(\xi)$, then
\begin{equation*}
 \int_{0}^{\xi}\frac{\omega(\eta)}{\eta^{\alpha}}\textrm{d} \eta\leq
 \delta +\frac{\omega(\xi)}{1-\alpha}\Big(\xi^{1-\alpha}-\delta^{1-\alpha}\Big)\leq
 \omega(\xi)\Big(\frac{2}{\alpha}+\frac{\xi^{1-\alpha}}{1-\alpha}\Big),
\end{equation*}
where the last inequality is due to
$\frac{\alpha}{2}\delta<\omega(\delta)\leq\omega(\xi)$ if $\delta$ is small enough (i.e. $\delta<(1-\frac{\alpha}{2})^{2/\alpha}$). Also
\begin{equation*}
 \int_{\xi}^{\infty}\frac{\omega(\eta)}{\eta^{1+\alpha}}\textrm{d}
 \eta=\frac{1}{\alpha}\frac{\omega(\xi)}{\xi^{\alpha}}+
 \frac{1}{\alpha}\int_{\xi}^{\infty}\frac{\gamma}{4\eta^{1+\alpha}}\textrm{d}\eta \leq
 \frac{1}{\alpha}\frac{\omega(\xi)}{\xi^{\alpha}}+\frac{1}{\alpha^2} \frac{ \gamma}{2}\frac{1}{\xi^{\alpha}}\leq
 \frac{2}{\alpha}\frac{\omega(\xi)}{\xi^{\alpha}}
\end{equation*}
if $\gamma< \alpha^2 \delta $ (thus $\gamma/2\leq \alpha \omega(\xi)$) and $\delta$ is small enough. Thus the
positive term is bounded from above by
$$A
\omega(\xi)\bigg(\frac{2}{\alpha}+\Big(\frac{1}{1-\alpha}+\frac{2}{\alpha}\Big)\xi^{1-\alpha}\bigg)\omega'(\xi)\leq
A\frac{\omega(\xi)}{\xi^{\alpha}}\frac{2}{\alpha(1-\alpha)}(\xi+\xi^{\alpha})\frac{\gamma}{4\xi}\leq
\frac{A \delta^{\alpha-1}\gamma}{\alpha(1-\alpha)}\frac{\omega(\xi)}{\xi^{\alpha}}.$$

For the negative part, we first observe that for $\xi\geq\delta$,
\begin{equation*}
 \omega(2\xi)=\omega(\xi)+\int_{\xi}^{2\xi}\omega'(\eta)\textrm{d}
 \eta=\omega(\xi)+\frac{\log2}{2} \gamma\leq \frac{3}{2}\omega(\xi),
\end{equation*}
under the same assumptions on $\delta$ and $\gamma$ as above. Also,
taking advantage of the concavity we obtain
$\omega(2\eta+\xi)-\omega(2\eta-\xi)\leq \omega(2\xi)$ for all
$\eta\geq\frac{\xi}{2}$. Therefore
\begin{equation*}
 \nu B \int_{\frac{\xi}{2}}^{\infty}\frac{\omega(2\eta+\xi)-\omega(2\eta-\xi)-2\omega(\xi)}{\eta^{1+\alpha}}\textrm{d}
 \eta\leq - \nu B \frac{\omega(\xi)}{2}\int_{\frac{\xi}{2}}^{\infty}\frac{1}{\eta^{1+\alpha}} \textrm{d}
 \eta = -\nu B\frac{2^{\alpha}}{2\alpha}\frac{\omega(\xi)}{\xi^{\alpha}}.
\end{equation*}
But $\frac{\omega(\xi)}{\xi^{\alpha}}(\frac{A\delta^{\alpha-1}
\gamma}{\alpha(1-\alpha)}-\nu B\frac{2^{\alpha}}{2\alpha})<0$ if $\gamma$ is small enough
(i.e. $\gamma< \min\{\alpha^2\delta,\frac{\nu (1-\alpha)B2^\alpha}{2  A} \delta^{1-\alpha}\}$).

\vskip0.2cm
\textbf{Case \textrm{II}:} when $\alpha\in[1,2[$

Precisely, we shall check the following inequality
\begin{equation*}
\begin{split}
 A \bigg[-\xi\Upsilon^\bot(\xi)+& \xi^{-\alpha+1}\omega(\xi)+
 \xi\int_{\xi}^{\infty}\frac{\omega(\eta)}{\eta^{1+\alpha}}\textrm{d}
 \eta\bigg]\omega'(\xi)
 +\nu B \int_{0}^{\frac{\xi}{2}}\frac{\omega(\xi+2\eta)+\omega(\xi-2\eta)-2\omega(\xi)}{\eta^{1+\alpha}}\textrm{d}
 \eta \\
 +& \nu B\int_{\frac{\xi}{2}}^{\infty}\frac{\omega(2\eta+\xi)-\omega(2\eta-\xi)-2\omega(\xi)}{\eta^{1+\alpha}}\textrm{d}
 \eta + \nu \Upsilon^\bot(\xi)<0 \quad \text{for all}\quad \xi>0.
\end{split}
\end{equation*}
We also further divide it into two cases.

\textbf{Case \textrm{II}.1:} $\alpha\in[1,2[$ and $0<\xi\leq \delta$
\vskip0.2cm
Since $\frac{\omega(\eta)}{\eta}\leq \omega'(0)=1$ for all $\eta>0$ and $-\Upsilon^\bot(\xi)\geq 0$, we have $-\xi\Upsilon^\bot(\xi)\leq -\delta\Upsilon^\bot(\xi)$
and $ \xi^{-\alpha+1}\omega(\xi)\leq \xi^{2-\alpha}$ and
\begin{equation*}
\int_{\xi}^{\delta}\frac{\omega(\eta)}{\eta^{1+\alpha}}\textrm{d}
 \eta \leq \int_{\xi}^{\delta}\frac{1}{\eta^{\alpha}}\textrm{d}
 \eta\leq
 \begin{cases}
   \frac{1}{\alpha-1} \xi^{1-\alpha},\quad& \alpha\in ]1,2[, \\
   \log(\delta/\xi),\quad &\alpha=1.
 \end{cases}  .
\end{equation*}
  Further, integration by parts leads to
\begin{equation*}
\begin{split}
 \int_{\delta}^{\infty}\frac{\omega(\eta)}{\eta^{1+\alpha}}\textrm{d}
 \eta & =\frac{1}{\alpha}\frac{\omega(\delta)}{\delta^{\alpha}}+
 \frac{1}{\alpha}\int_{\delta}^{\infty}\frac{  \gamma}{4\eta^{\alpha+1}}\textrm{d}
 \eta \\
 & \leq \frac{1}{\alpha} \frac{1}{\delta^{\alpha-1}}+
  \frac{ \gamma}{4\alpha^2}\frac{1}{\delta^{\alpha}}\leq 2 \frac{1}{\delta^{\alpha-1}}\leq  2 \xi^{1-\alpha}.
\end{split}
\end{equation*}
Clearly $\omega'(\xi)\leq\omega'(0)=1$, so we get that the positive
part is bounded by
\begin{equation*}
\begin{cases}
 A \big(-\delta \Upsilon^\bot(\xi)+ \xi^{2-\alpha}\frac{4}{\alpha-1}\big),\quad &\alpha\in]1,2[, \\
 A \big(-\delta \Upsilon^\bot(\xi)+ \xi(3+\log\frac{\delta}{\xi})\big),\quad &\alpha=1.
\end{cases}
\end{equation*}

For the negative part, we have
\begin{equation*}
\begin{split}
 \nu B \int_{0}^{\frac{\xi}{2}}&\frac{\omega(\xi+2\eta)+\omega(\xi-2\eta)-2\omega(\xi)}{\eta^{1+\alpha}}\textrm{d}
 \eta \leq
 \nu B\int_{0}^{\frac{\xi}{2}}\frac{\omega''(\xi)2\eta^{2}}{\eta^{1+\alpha}}\textrm{d}
 \eta\\  =  &
 -\nu B\frac{\alpha(2+\alpha)}{2^{1-\alpha}(2-\alpha)}\xi^{1-\frac{\alpha}{2}}\leq
 -3\nu B\xi^{1-\frac{\alpha}{2}}.
\end{split}
\end{equation*}
But, clearly if $\delta$ is chosen small enough, we find that for every $\xi\in ]0,\delta]$
\begin{equation*}
\begin{cases}
  (-A\delta+\nu )\Upsilon^\bot(\xi)+ \xi^{2-\alpha}\big(A \frac{4}{\alpha-1}-3\nu B
  \xi^{\frac{\alpha}{2}-1}\big)<0, \quad & \alpha\in ]1,2[,\\
  (-A\delta+\nu )\Upsilon^\bot(\xi)+ \xi\big(3A +A \log\frac{\delta}{\xi}-3\nu B
  \xi^{-\frac{1}{2}}\big)<0, \quad & \alpha=1.
\end{cases}
\end{equation*}

\vskip0.2cm
 \textbf{Case \textrm{II}.2:} $\alpha\in [1,2[$ and $\xi\geq \delta$
\vskip0.2cm

For the positive part we have
\begin{equation*}
\begin{split}
 \int_{\xi}^{\infty}\frac{\omega(\eta)}{\eta^{1+\alpha}}\textrm{d}
 \eta & =\frac{1}{\alpha}\frac{\omega(\xi)}{\xi^{\alpha}}+
 \frac{1}{\alpha}\int_{\xi}^{\infty}\frac{ \gamma}{4\eta^{\alpha+1}}\textrm{d}
 \eta \\
 & \leq \frac{1}{\alpha} \frac{\omega(\xi)}{\xi^{\alpha}}+   \frac{ \gamma}{4\alpha^2}\frac{1}{\xi^{\alpha}}\leq 2
 \frac{\omega(\xi)}{\xi^{\alpha}},
\end{split}
\end{equation*}
where we have used the simple fact that $\gamma\leq\frac{\delta}{2}\leq \omega(\delta)\leq \omega(\xi)$.
Thus the positive term is bounded from above by
\begin{equation*}
  A \big(-\xi\Upsilon^\bot(\xi)+ 3\omega(\xi) \xi^{1-\alpha}\big)\omega'(\xi)
  =A \big(-\xi\Upsilon^\bot(\xi)+ 3\omega(\xi) \xi^{1-\alpha}\big)\frac{\gamma}{4\xi}\leq
  -A\gamma\Upsilon^\bot(\xi)+ A\gamma \frac{\omega(\xi)}{\xi^\alpha}.
\end{equation*}

For the negative part, in a similar way as treating the corresponding part in case I.2, we have
\begin{equation*}
 \nu B \int_{\frac{\xi}{2}}^{\infty}\frac{\omega(2\eta+\xi)-\omega(2\eta-\xi)-2\omega(\xi)}{\eta^{1+\alpha}}\textrm{d}
 \eta\leq - \nu B \frac{\omega(\xi)}{2}\int_{\frac{\xi}{2}}^{\infty}\frac{1}{\eta^{1+\alpha}} \textrm{d}
 \eta \leq -\frac{\nu B}{2}\frac{\omega(\xi)}{\xi^{\alpha}}.
\end{equation*}
But, clearly $(-A\gamma+\nu)\Upsilon^\bot(\xi)+ \frac{\omega(\xi)}{\xi^{\alpha}}(A\gamma-\frac{\nu B }{2})<0$ if
$\gamma$ is small enough.

Therefore both case $\textrm{I}$ and case $\textrm{II}$ yield $f'(T_{*})<0$.

Finally, only case \eqref{case a} occurs and we obtain $T^{*}=\infty$. Moreover
\begin{equation*}
\begin{aligned}
  \norm{\nabla\theta(t)}_{L^{\infty}} < \lambda,\qquad  \forall t\in [0,\infty[,
\end{aligned}
\end{equation*}
where the value of $\lambda\sim C\|\nabla\theta_0\|_{L^\infty} e^{C\|\theta_0\|_{L^\infty}} $ is given by \eqref{eq lambdValue}.

\begin{proof}[Proof of Lemma \ref{lem fact}]
 Set $C':=\omega^{-1}(3\norm{\theta_0}_{L^\infty})$, then from the maximum principle \eqref{eq TDMaxPrin}, we get
 \begin{equation}
  \lambda|x-y|\geq C' \Rightarrow |\theta(T_*,x)-\theta(T_*,y)|<\frac{2}{3}\omega_\lambda(|x-y|).
 \end{equation}
 Since $\nabla\theta(t)\in \mathcal{C}([0,T^*),H^{m-1}(\mathbb{R}^2))$, then for every $\epsilon>0$, there exists $R>0$ such that
 \begin{equation*}
  \norm{\nabla \theta(T_*)}_{L^\infty(\mathbb{R}^2 \setminus B_{R})}\leq C_0 \norm{\nabla\theta(T_*)}_{H^{m-1}(\mathbb{R}^2\setminus B_{R})}\leq \epsilon,
 \end{equation*}
 where $B_{R}$ is a ball centered at the origin with the radius $R$ and $\mathbb{R}^2\setminus B_R$ is its complement. Thus for every $x,y$($x\neq y$) satisfying
 that $\lambda|x-y|\leq C' $ and $x$ or $y$ belongs to $\mathbb{R}^2\setminus B_{R+C'/\lambda}$, we get
 \begin{equation*}
  |\theta(T_*,x)-\theta(T_*,y)|\leq \norm{\nabla \theta(T_*)}_{L^\infty(\mathbb{R}^2\setminus B_{R})}|x-y|\leq \epsilon |x-y|.
 \end{equation*}
 Taking advantage of the following inequality from the concavity of $\omega$
 \begin{equation*}
  \frac{\omega(C')}{C'}\lambda |x-y| \leq \omega_\lambda(|x-y|),
 \end{equation*}
 we can take $\epsilon$ small enough such that $\epsilon<\frac{1}{2} \frac{\omega(C')}{C'}\lambda$ to obtain
 \begin{equation}
  \lambda|x-y|\leq C', \, x\, \mathrm{or} \,y \in \mathbb{R}^2\setminus B_{R+\frac{C'}{\lambda}};\Rightarrow |\theta(T_*,x)-\theta(T_*,y)|<\frac{1}{2}\omega_\lambda(|x-y|).
 \end{equation}
 Now it remains to consider the case when $x,y\in B_{R+\frac{C'}{\lambda}}$. From the smoothing effect we know $\norm{\nabla^2
 \theta(T_*)}_{L^\infty}<\infty$, thus we have (cf. \cite{KisNV})
 \begin{equation*}
  \norm{\nabla \theta(T_*)}_{L^\infty(B_{R+\frac{C'}{\lambda}})}<\lambda\omega'(0).
 \end{equation*}
 Let $\delta'\ll 1$ small enough, then we see
 \begin{equation*}
  \norm{\theta(T_*)}_{L^\infty(B_{R+\frac{C'}{\lambda}})}< \lambda (1-\delta')\frac{\omega(\delta')}{\delta'}.
 \end{equation*}
 Thus for every $x,y$($x\neq y$) satisfying that $\lambda|x-y|\leq \delta'$ and both $x,y$ belongs to $B_{R+C'/\lambda}$, we have
 \begin{equation}
 \begin{split}
  |\theta(T_*,x)-\theta(T_*,y)| & \leq \norm{\nabla \theta(T_*)}_{L^\infty(B_{R+\frac{C'}{\lambda}})}|x-y| \\
  & < (1-\delta')\frac{\omega(\delta')}{\delta'}\lambda|x-y|
  \leq (1-\delta')\omega_\lambda(|x-y|).
 \end{split}
 \end{equation}
 We set
 \begin{equation*}
  \Omega:=\{(x,y)\in \mathbb{R}^2\times \mathbb{R}^2 : \max\{|x|,|y|\}\leq R+\frac{C'}{\lambda},\,  |x-y|\geq \frac{\delta'}{\lambda}\},
 \end{equation*}
 then from the above results we necessarily have
 \begin{equation*}
  1=\sup_{x\neq y}\frac{|\theta(T_*,x)-\theta(T_*,y)|}{\omega_\lambda(|x-y|)}=\sup_{(x,y)\in\Omega}\frac{|\theta(T_*,x)-\theta(T_*,y)|}{\omega_\lambda(|x-y|)}.
 \end{equation*}
 Thus the conclusion follows from the compactness of $\Omega$.

For \eqref{eq fact0}, it is from a direct computation under the scenario \eqref{eq fact} (cf. Proposition 2.4 in \cite{Kis}).
\end{proof}

\section{Appendix}
\setcounter{section}{6}\setcounter{equation}{0}

\subsection{The formula for $\mathcal{R}_{\alpha,j}f$}
\begin{proof}[Proof of Proposition \ref{prop Fmula}]
 The pseudo-differential operator $\mathcal{R}_{\alpha,j}$ ($\alpha\in]0,2[$) is the composition of two operators $|D|^{\alpha-1}$ and $\mathcal{R}_{j}$, which both are (constant
 coefficient) pseudo-differential operators, thus the symbol of $\mathcal{R}_{\alpha,j}$ is $-i \zeta_j/|\zeta|^{2-\alpha}$. Now we want to know
 the explicit formula of $\mathcal{F}^{-1}
 (-i \zeta_j/|\zeta|^{2-\alpha})$.

 From the equality in the distributional sense
 \begin{equation*}
  \frac{\partial}{\partial x_j} |x|^{-(n+\alpha-2)} = -(n+\alpha-2) \mathrm{p.v.} \frac{x_j}{|x|^{n+\alpha}},
 \end{equation*}
 and the known formula that for every $0<a<n$ (c.f. \cite{ref Duo})
 \begin{equation*}
  (|x|^{-a})^{\wedge}(\zeta)=\frac{2^{n-a}\pi^{n/2}\Gamma(\frac{n-a}{2})}{\Gamma(\frac{a}{2})} |\zeta|^{-n+a},
 \end{equation*}
 we directly have
 \begin{equation*}
 \begin{split}
  (\mathrm{p.v.} \frac{x_j}{|x|^{n+\alpha}})^{\wedge}(\zeta) & = - \frac{1}{n+\alpha-2} (\partial_{x_j} |x|^{-n-\alpha+2})^{\wedge}(\zeta) \\
  & =-\frac{i \zeta_j}{n+\alpha-2} (|x|^{-n-\alpha+2})^{\wedge}(\zeta) \\
  & =- \frac{i \zeta_j}{n+\alpha-2} \frac{2^{2-\alpha}\pi^{n/2}\Gamma(\frac{2-\alpha}{2})}{\Gamma(\frac{n+\alpha-2}{2})} |\zeta|^{\alpha-2} \\
  & =- i \frac{2^{1-\alpha}\pi^{n/2}\Gamma(\frac{2-\alpha}{2})}{\Gamma(\frac{n+\alpha}{2})}\cdot \frac{\zeta_j}{|\zeta|^{2-\alpha}}.
 \end{split}
 \end{equation*}
 \end{proof}

\subsection{A commutator estimate}
The key to the proof of the uniform estimate is the following
commutator estimate
\begin{lemma}\label{lem commutator}
Let $v$ be a divergence free vector field over $\mathbb{R}^{n}$. For every
$q\in\mathbb{N}$, denote
\begin{equation*}
F_{q}(v,f):=S_{q+1}v\cdot\nabla\Delta_{q}f-\Delta_{q}(v\cdot\nabla
f).
\end{equation*}
Then for every $\beta\in]0,1[$, there exists a positive constant $C$ such that
\begin{equation}\label{eq commutator1}
\begin{split}
 & 2^{-q\beta}\norm{F_{q}(v,f)}_{L^2}
 \\ \leq& C\norm{|D|^{1-\beta}
 v}_{L^{\infty}} \Big(\sum_{q'\leq q+4}2^{q'-q}\norm{\Delta_{q'}f}_{L^{2}} + \sum_{q'\geq q-4} 2^{(q-q')(1-\beta)}\norm{\Delta_{q'}f}_{L^2}\Big),
\end{split}
\end{equation}
\\
Especially, in the case $n=2$ and $v=|D|^{\alpha-1}\mathcal{R}^{\bot}f$ ($\alpha\in]0,2[$), we
further have for every $\beta\in\big]\max\{0,\alpha-1\},1\big[$ and every $q\in\mathbb{N}$
\begin{equation}\label{eq commutator2}
\begin{split}
 & 2^{-q\beta}\norm{F_{q}(v,f)}_{L^{2}} \\ \leq & C \Big(\norm{|D|^{1-\beta}
 v}_{L^{\infty}}\sum_{q'\geq q-4}2^{(q-q')(1-\beta)}\norm{\Delta_{q'}f}_{L^{2}}+
 \norm{|D|^{\alpha-\beta}f}_{L^{\infty}}\sum_{|q'-q|\leq
 4}\norm{\Delta_{q'}f}_{L^{2}}\Big).
\end{split}
\end{equation}
Moreover, when $\beta=0$, $\alpha\in ]0,1[$, \eqref{eq commutator1} and \eqref{eq commutator2} hold if we replace $\norm{|D|^{1-\beta}v}_{L^\infty}$ by $\norm{\nabla v}_{L^\infty}$;
and when $\beta=1$, $\alpha=2$, then \eqref{eq commutator1} and \eqref{eq commutator2} hold if we make such a modification
$$\norm{|D|^{1-\beta}v}_{L^\infty}\rightarrow\norm{v}_{B^0_{\infty,1}},\quad \norm{|D|^{\alpha-\beta}f}_{L^2}\rightarrow\norm{\nabla f}_{L^\infty}.$$
\end{lemma}

\begin{proof}
Using Bony decomposition, we decompose $F_{q}(v,f)$ into $\sum_{i=1}^{6}F^{i}_{q}(v,f)$, where
\begin{equation*}
 F_{q}^{1}(v,f)=(S_{q+1}v-v)\cdot\nabla\Delta_{q}f, \quad
 F_{q}^{2}(v,f)=[\Delta_{-1}v,\Delta_q]\cdot\nabla f,
\end{equation*}
\begin{equation*}
 F_{q}^{3}(v,f)=\sum_{q'\in\mathbb{N}}[S_{q'-1}\widetilde{v},\Delta_{q}]\cdot\nabla\Delta_{q'}f, \quad
 F_{q}^{4}(v,f)=\sum_{q'\geq -1}\Delta_{q'}\widetilde{v}\cdot\nabla \Delta_{q}S_{q'+2}f,
\end{equation*}
\begin{equation*}
 F_{q}^{5}(v,f)= -\sum_{q'\in\mathbb{N}}\Delta_{q}\Big(\Delta_{q'}\widetilde{v}\cdot\nabla S_{q'-1}f\Big), \quad
 F_{q}^{6}(v,f)=- \sum_{q'\geq -1}\mathrm{div}\Delta_{q}\Big(\Delta_{q'}\widetilde{v}\sum_{i\in\{\pm1,0\}}\Delta_{q'+i}f\Big),
\end{equation*}
where $[A,B]:=AB-BA$ denotes the commutator operator and
$\widetilde{v}:=v-\Delta_{-1}v$ denotes the high frequency part of $v$.
\\
For $F_{q}^{1}$, from the divergence-free property of $v$ we directly obtain that when $1-\beta>0$
\begin{equation*}\label{eq Fq1}
\begin{split}
 2^{-q\beta}\norm{F_{q}^{1}(v,f)}_{L^{2}}
  & \lesssim \sum_{q'\geq q+1}2^{(1-\beta)(q-q')} 2^{q'(1-\beta)}\norm{\Delta_{q'}v}_{L^\infty}\norm{\Delta_q f}_{L^2} \\ &\lesssim
 \norm{|D|^{1-\beta}v}_{L^\infty}\norm{\Delta_q f}_{L^2}.
\end{split}
\end{equation*}
For $F_q^2$, since $F^q_2(v,f)=\sum_{|q'-q|\leq 1} [\Delta_{-1}v,\Delta_{q}]\cdot\nabla\Delta_{q'} f$, then from the expression formula of
$\Delta_{q}$ and mean value theorem, we get that when $\beta>0$
\begin{equation*}\label{eq Fq2}
\begin{split}
 2^{-q\beta}\norm{F_{q}^{2}(v,f)}_{L^2}
 & \lesssim 2^{-q\beta} 2^{-q} \norm{\nabla\Delta_{-1}v}_{L^\infty}\sum_{|q'-q|\leq1}2^{q'}\norm{\Delta_{q'}f}_{L^{2}}
 \\ & \lesssim \sum_{-\infty\leq j\leq-1}2^{j\beta}\norm{|D|^{1-\beta}\dot \Delta_j v}_{L^\infty}\sum_{|q'-q|\leq1}\norm{\Delta_{q'}f}_{L^{2}}
 \\ &\lesssim \norm{|D|^{1-\beta} v}_{L^\infty}\sum_{|q'-q|\leq1}\norm{\Delta_{q'}f}_{L^{2}}.
\end{split}
\end{equation*}
For $F_q^3$, similarly as estimating $F_q^2$, we infer
\begin{equation*}\label{eq Fq3}
\begin{split}
 2^{-q\beta}\norm{F_{q}^{3}(v,f)}_{L^2}
 & \lesssim 2^{-q\beta}\sum_{|q'-q|\leq 4} 2^{-q}\norm{\nabla
 S_{q'-1}\widetilde{v}}_{L^\infty} 2^{q'}\norm{\Delta_{q'}f}_{L^{2}}
 \\ &\lesssim \sum_{|q'-q|\leq4} \sum_{q''\leq q'-2} 2^{(q''-q')\beta}\norm{|D|^{1-\beta}\Delta_{q''}\widetilde{v}}_{L^\infty}\norm{\Delta_{q'}f}_{L^2} \\
 &\lesssim \norm{|D|^{1-\beta} v}_{L^\infty}\sum_{|q'-q|\leq 4}\norm{\Delta_{q'}f}_{L^{2}}.
\end{split}
\end{equation*}
For $F^4_q$ and $F^5_q$, from the spectral property and the fact
$2^{q'(1-\beta)}\norm{\Delta_{q'}\widetilde{v}}_{L^\infty}\approx
\norm{\Delta_{q'}|D|^{1-\beta}\widetilde{v}}_{L^\infty}$, we have
\begin{equation*}\label{eq Fq4}
\begin{split}
 2^{-q\beta}\norm{F^4_q(v,f)}_{L^2}\lesssim \sum_{q'\geq q-2} 2^{(q-q')(1-\beta)} 2^{q'(1-\beta)}\norm{\Delta_{q'}\widetilde{v}}_{L^\infty}
 \norm{\Delta_q f}_{L^2}\lesssim \norm{|D|^{1-\beta}
 v}_{L^\infty}\norm{\Delta_q f}_{L^2}.
\end{split}
\end{equation*}
\begin{equation*}\label{eq Fq5}
\begin{split}
 2^{-q\beta}\norm{F_q^5(v,f)}_{L^2} & \lesssim 2^{-q\beta}\sum_{|q'-q|\leq 4}2^{q'}\norm{\Delta_{q'} \widetilde{v}}_{L^\infty}
 \sum_{q''\leq q'-2}2^{q''-q'}\norm{\Delta_{q''}f}_{L^2} \\
 & \lesssim \norm{|D|^{1-\beta} v}_{L^\infty} \sum_{q''\leq q+2}2^{q''-q}\norm{\Delta_{q''}f}_{L^2}.
\end{split}
\end{equation*}
Besides, for $F^5_q$ when $v=|D|^{\alpha-1}\mathcal{R}^{\bot}f$, we alteratively
have the following improvement that when $\beta>\alpha-1$
\begin{equation*}\label{eq Fq5.1}
\begin{split}
 2^{-q\beta}\norm{F_{q}^{5}(v,f)}_{L^{2}} & \leq 2^{-q\beta}\sum_{|q'-q|\leq
 4}\norm{\Delta_{q'}(Id-\Delta_{-1})|D|^{\alpha-1}\mathcal{R}^{\bot}f}_{L^{2}}
 \norm{\nabla S_{q'-1}  f}_{L^{\infty}} \\ & \lesssim \sum_{|q'-q|\leq
 4}\norm{\Delta_{q'}f}_{L^{2}}\sum_{-\infty\leq q''\leq q'-2} 2^{(\alpha-1-\beta)(q'-q'')}
 \norm{ |D|^{\alpha-\beta}\dot{\Delta}_{q''} f}_{L^{\infty}} \\ &
 \lesssim
 \norm{|D|^{\alpha-\beta}
 f}_{L^{\infty}}\sum_{|q'-q|\leq 4}
 \norm{\Delta_{q'}f}_{L^{2}}.
\end{split}
\end{equation*}
Finally, for $F^6_q$ we easily have
\begin{equation*}\label{eq Fq6}
\begin{split}
 2^{-q\beta}\norm{F^6_q(v,f)}_{L^2}& \lesssim \sum_{q'\geq q-3} 2^{(q-q')(1-\beta)}\, 2^{q'(1-\beta)}\norm{\Delta_{q'}\widetilde{v}}_{L^\infty}
 \sum_{i\in\{\pm1,0\}}\norm{\Delta_{q'+i}f}_{L^2} \\
 & \lesssim \norm{|D|^{1-\beta} v}_{L^\infty}\sum_{q'\geq q-4}2^{(q-q')(1-\beta)}\norm{\Delta_{q'} f}_{L^2}.
\end{split}
\end{equation*}
Combining the above estimates appropriately yields the inequalities
\eqref{eq commutator1} and \eqref{eq commutator2}.

\end{proof}

\subsection{Proof of Lemma \ref{lem Modulidiffer}}
\begin{proof}
We treat the general $n$-dimensional case. Let
$x=(x_{1},\tilde{x})=(x_{1},x_{2},\cdots,x_{n})$ and the Fourier
variable
$\zeta=(\zeta_{1},\tilde{\zeta})=(\zeta_{1},\zeta_{2},\cdots,\zeta_{n})$.
First we observe that for every $\alpha\in]0,2[$ (cf. \cite{ref CAFS})
\begin{equation*}
(-|D|^{\alpha})\theta=\frac{d}{dh}e^{-h|D|^{\alpha}}\theta\Big|_{h=0}=\frac{d}{dh}\mathcal{P}^{\alpha}_{h,n}*\theta\Big|_{h=0}
\end{equation*}
where
\begin{equation*}
\mathcal{P}^{\alpha}_{h,n}(x): =c'_{n,\alpha}\frac{h}{(|x|^{2}+\alpha^2 h^{\frac{2}{\alpha}})^{\frac{n+\alpha}{2}}}
\end{equation*}
and $c'_{n,\alpha}$ is the normalization constant such that
$\int\mathcal{P}^{\alpha}_{h,n}\textrm{d}
 x=1(=e^{-h|\zeta|^{\alpha}}|_{\zeta=0})$. In the following we take $\mathcal{P}_{h,n}$ instead
of $\mathcal{P}^{\alpha}_{h,n}$ for brevity. Thus our task reduces
to estimate
\begin{equation*}
(\mathcal{P}_{h,n}*\theta)(x)-(\mathcal{P}_{h,n}*\theta)(y).
\end{equation*}
Due to the translation and rotation invariant properties, we may
assume that $x=(\frac{\xi}{2},0,\cdots,0)$ and
$y=(-\frac{\xi}{2},0,\cdots,0)$. Then from the symmetry and
monotonicity of the kernel $\mathcal{P}_{h,n}$ and the fact
$$\int_{\mathbb{R}^{n-1}}\mathcal{P}_{h,n}(x_{1},\tilde{x})\textrm{d}
 \tilde{x}=\mathcal{F}^{-1}(\widehat{\mathcal{P}_{h,n}}|_{\tilde{\zeta}=0})(x_1)
 =\mathcal{F}^{-1}(e^{-h|\zeta_{1}|^{\alpha}})(x_1)=\mathcal{P}_{h,1}(x_{1})$$
we have
\begin{equation*}
\begin{split}
 &(\mathcal{P}_{h,n}*\theta)(x)-(\mathcal{P}_{h,n}*\theta)(y)
 \\
 &=\iint_{\mathbb{R}^{n}}\big[\mathcal{P}_{h,n}\bigl(\frac{\xi}{2}-\eta,-\tilde{\eta}\bigr)-
 \mathcal{P}_{h,n}\bigl(-\frac{\xi}{2}-\eta,-\tilde{\eta}\bigl)\big]
 \theta(\eta,\tilde{\eta})\textrm{d} \eta \textrm{d}\tilde{\eta}
 \\
 &= \int_{\mathbb{R}^{n-1}}\textrm{d} \tilde{\eta}\int_{0}^{\infty}
 \big[\mathcal{P}_{h,n}\bigl(\frac{\xi}{2}-\eta,\tilde{\eta}\bigr)-
 \mathcal{P}_{h,n}\bigl(-\frac{\xi}{2}-\eta,\tilde{\eta}\bigl)\big]
 \bigl[\theta(\eta,\tilde{\eta})-\theta(-\eta,\tilde{\eta})\bigr]\textrm{d} \eta
 \\
 & \leq  \int_{\mathbb{R}^{n-1}}\textrm{d} \tilde{\eta}\int_{0}^{\infty}
 \big[\mathcal{P}_{h,n}\bigl(\frac{\xi}{2}-\eta,\tilde{\eta}\bigr)-
 \mathcal{P}_{h,n}\bigl(-\frac{\xi}{2}-\eta,\tilde{\eta}\bigl)\big]
 \omega(2\eta)\textrm{d} \eta
 \\
 & =\int_{0}^{\infty}
 \big[\mathcal{P}_{h,1}\bigl(\frac{\xi}{2}-\eta\bigr)-
 \mathcal{P}_{h,1}\bigl(-\frac{\xi}{2}-\eta\bigl)\big]
 \omega(2\eta)\textrm{d} \eta
 \\
 & =\int_{0}^{\frac{\xi}{2}}\mathcal{P}_{h,1}(\eta)\bigl[\omega(2\eta+\xi)+\omega(\xi-2\eta)\bigr]\textrm{d}\eta
 +\int_{\frac{\xi}{2}}^{\infty}
 \mathcal{P}_{h,1}(\eta)\bigl[\omega(2\eta+\xi)-\omega(2\eta-\xi))\bigr]\textrm{d}\eta
\end{split}
\end{equation*}
Because of $\int_{0}^{\infty}\mathcal{P}_{h,1}(\eta)\textrm{d}
\eta=\frac{1}{2}$, we have the estimate of the difference
\begin{equation*}
\begin{split}
 (\mathcal{P}_{h,n}*\theta)(x)-&(\mathcal{P}_{h,n}*\theta)(y)-\omega(\xi)
 \\
 \leq & \int_{0}^{\frac{\xi}{2}}\mathcal{P}_{h,1}(\eta)\bigl[\omega(2\eta+\xi)+\omega(\xi-2\eta)-2\omega(\xi)\bigr]\textrm{d}
 \eta
 \\  &+ \int_{\frac{\xi}{2}}^{\infty}
 \mathcal{P}_{h,1}(\eta)\bigl[\omega(2\eta+\xi)-\omega(2\eta-\xi)-2\omega(\xi)\bigr]\textrm{d} \eta
\end{split}
\end{equation*}
Hence from the above estimates and the explicit formula of kernel
$\mathcal{P}_{h,1}$, we can conclude that
\begin{equation*}
\begin{split}
 &\bigl[(-|D|^{\alpha})\theta\bigr](x)-\bigl[(-|D|^{\alpha})\theta\bigr](y)
 \\
 &=
 \lim_{h\rightarrow 0}\frac{[(\mathcal{P}_{h,n}*\theta)(x)-\theta(x)]-[(\mathcal{P}_{h,n}*\theta)(y)-\theta(y)]}{h}
 \\
 &=\lim_{h\rightarrow
 0}\frac{(\mathcal{P}_{h,n}*\theta)(x)-(\mathcal{P}_{h,n}*\theta)(y)-\omega(\xi)}{h}
 \\
 & \lesssim_{\alpha,n} \int_{0}^{\frac{\xi}{2}}\frac{\omega(\xi+2\eta)+\omega(\xi-2\eta)-2\omega(\xi)}{\eta^{1+\alpha}}\textrm{d}
 \eta
  +\int_{\frac{\xi}{2}}^{\infty}\frac{\omega(2\eta+\xi)-\omega(2\eta-\xi)-2\omega(\xi)}{\eta^{1+\alpha}}\textrm{d}
 \eta
\end{split}
\end{equation*}

\end{proof}


\textbf{Acknowledgments:} The authors would like to thank Prof.
P.Constantin for helpful advice and discussion. They would also like to express their deep gratitude to the anonymous referees
for their kind suggestions. The authors were partly supported by the NSF of China (No.10725102).


\end{document}